\documentclass[11pt]{article}
\usepackage{amsmath,amsthm,amsfonts,amssymb,amscd, amsxtra}
\newcommand{\banacha}{\mathbb X}
\newcommand{\banachb}{\mathbb Y}
\newcommand{\banachc}{\mathbb Z}
\newtheorem{theorem}{Theorem}
\newtheorem{lemma}[theorem]{Lemma}

\newtheorem{corollary}[theorem]{Corollary}
\newtheorem{proposition}[theorem]{Proposition}

\newtheorem{remark}{Remark}

\begin{document} 
\title{ A robust  semi-local  convergence analysis   of Newton's  method  for cone inclusion problems in Banach spaces under   affine invariant   majorant condition}

\author{ O.  P. Ferreira\thanks{IME/UFG,  CP-131, CEP 74001-970 - Goi\^ania, GO, Brazil (Email: {\tt
      orizon@ufg.br}). The author was supported in part by CAPES (Projeto  019/2011- Coopera\c c\~ao Internacional Brasil-China), CNPq Grant 471815/2012-8, CNPq Grant 303732/2011-3, PRONEX--Optimization(FAPERJ/CNPq) and FAPEG/GO.} 
}
\date{March 10, 2014}  
\maketitle
\begin{abstract}
A  semi-local  analysis of Newton's method for solving nonlinear inclusion problems in Banach space  is presented in this paper. Under a affine majorant condition on the nonlinear function which is associated to the inclusion problem, the robust convergence of the method and results on the convergence rate are established. Using this result we show that the robust analysis of the Newton's method for solving nonlinear inclusion problems under affine Lipschitz-like and  affine Smale's  conditions can be obtained as a special case of the general theory. Besides for the degenerate cone,  which the nonlinear inclusion  becomes  a nonlinear equation, ours analysis retrieve the classical results on local analysis of Newton's method.\\


\end{abstract}
\section{Introduction}\label{sec:int}
The idea of solving a nonlinear  inclusion problems of the form
\begin{equation} \label{eq:ipi}
\mbox {Find} ~ \bar{x} ~\mbox{such that} ~  F(\bar{x}) \in C, 
\end{equation}
where $C$  is a nonempty closed convex cone in a Banach space $\banachb$ and $F$ is a function from a Banach space $\banacha$ into $\banachb$,  plays a huge role in classical analysis and its applications. For instance,  the  special case in which $C$ is the degenerate cone $\{0\}\subset \banachb$,  the  inclusion  problem  in  \eqref{eq:ipi} correspond to a nonlinear equation.   In the case for which $\banacha=\mathbb{R}^{n}$, $\banachb=\mathbb{R}^{p+q}$ and $C=\mathbb{R}^{p}_{-}\times \{0\}$ is the product of the nonpositive orthant in $\mathbb{R}^{p}$ with the origin in  $\mathbb{R}^{q}$,  the  inclusion  problem  in  \eqref{eq:ipi} correspond to a nonlinear systems of  $p$  inequalities and  $q$ equality,  for example see \cite{BCSS97}, \cite{DS96},     \cite{D73}, \cite{DHB}, \cite{DR2009}, \cite{HeSun2005},  \cite{KP02}, \cite{LiNg2012}  and  \cite{P70}.  

The Newton's  method and its variant are powerful tools for solving nonlinear equation in  Banach space, having a wide range of applications in pure and applied mathematics, see \cite{B1994},  \cite{BCSS97}, \cite{DHB},  \cite{DR2010},  \cite{DR2009}, \cite{KP02},  \cite{M61},  \cite{N56} and \cite{S199}. Newton's Method has  been extended in order  to solve nonlinear systems of equalities and inequalities (see \cite{D73}, \cite{Rob1972-2}, \cite{P70}) and several  other different purposes, see  \cite{ADMMS02}, \cite{ABM2004}, \cite{B1994}, \cite{BCSS97}, \cite{DPM03}, \cite{DR2010},  \cite{DR2009},  \cite{WL09}, \cite{LW07}, \cite{LW06}, \cite{S199} and \cite{WHL2009}.  In particular,  Robinson in \cite{Rob1972-2} (see also \cite{LiNg2014}, \cite{LiNg2012}, \cite{WHL2009})  generalized Newton's method for solving problems of the form  \eqref{eq:ipi},  which becomes  the usual Newton's method  to the special case in which $C$ is the degenerate cone $\{0\}\subset \banachb$.
 
 One of the usual hypotheses used in convergence analysis of Newton's method is that the Lipschitz
continuity of the derivative of the nonlinear operator in question or something like Lipschitz continuity
is critical; that is, keeping control of the derivative is an important point in the convergence analysis
of Newton's method.   These assumptions seem  to be actually natural in analysis of Newton's method, \cite{C10},  \cite{DR2010}, \cite{DR2009}, \cite{F08},  \cite{FS09},  \cite{Hu04}, \cite{Cl08}, \cite{MR2475307} and \cite{W00}.  On the other hand, in the last few years, a couple of papers have dealt with the issue of  convergence analysis of the Newton's method and its variants  by relaxing the assumption of Lipschitz continuity of the derivative of the function, which define  the nonlinear equation in consideration, see \cite{C10}, \cite{F11}, \cite{F08},   \cite{FGO13}, \cite{FS12},  \cite{FS09}, \cite{FS2002},  \cite{Hu04}, \cite{LiNg2012}, \cite{WL09},  \cite{Li2009}, \cite{Cl08},  \cite{LW07}, \cite{LW06},  \cite{MR2475307}, \cite{W99} and  \cite{W00}.  In particular,  in  \cite{FS09}, under a affine majorant condition,  a  semi-local convergence, as well as  results on the convergence rate are established. The advantage of  working with a  affine majorante condition rests in the fact that it allow to unify  converge results on Newton's method without previous relationship, see  \cite{FS09},  besides makes the results insensitive with respect to invertible continuous linear mapping, see \cite{DHB} and \cite{DH}. It is worth mentioning that the affine majorant condition used  in  \cite{FS09} is equivalent to Wang's condition (see  \cite{W99}) whenever   the derivative of majorant function is convex.  

In  \cite{FS2002} a new technique for analyzing the convergence properties  of Newton's Method which simplifies the analyses and proof of the results has been introduced.  After that, this technique was successfully employed for analyzing Newton's method in difference context, \cite{ABM2004}, \cite{FS09}, \cite{WL09},  \cite{Li2009}, \cite{LW07} and  \cite{LW06}. In the present work, we will use the technique introduced in  \cite{FS2002} to present a  semi-local convergence analysis of Newton's method  for  solving a nonlinear  inclusion problems of the form \eqref{eq:ipi}.  In our analysis, the classical  Lipschitz continuity of the derivative of the nonlinear function, which define  the nonlinear inclusion in consideration,  is relaxed by using a affine majorant function.  Although the semi-local convergence analysis of the Newton method  for solving nonlinear inclusion  under  Lipschitz-like (see,  \cite{FGO13},   \cite{LiNg2012} and  \cite{Rob1972-2}) or  Wang's   condition (see \cite{LiNg2014}, \cite{LiNg2007} and \cite{WHL2009})    are well known, as far as we know,   the {\it robust} semi-local convergence analysis  of the Newton's method   for solving nonlinear inclusion problems in Banach space  under general affine invariant assumptions  is a new contribution of this paper. 

 The basic idea of the analysis is to find a good region for Newton's Method.  In this region, the  majorant function indeed bounds the nonlinear function  associated to the inclusion problem, the behavior of Newton's method is controlled by Newton's iteration of the majorant  function and, as a consequence,   the region  is  invariant under Newton's iteration  multifunction. Besides,  in this  region,   the analysis  provides a clear relationship between the majorant function and the nonlinear function  under consideration and allow us to obtain bounds for the $Q$-quadratic convergence of the  method,  which depend on the behavior  of the majorant function at its smallest zero.  It is worth pointing out that, the technique employed also makes it possible to analyze the method in the presence of errors in the initial point, which shows that the assumption on the initial point is robust.  Finally, using this convergence result we obtain a robust analysis of the Newton's method for solving nonlinear inclusion problems under affine invariant Lipschitz-like and Smale's conditions as a special case of the general theory. We remark that,  for the special case in which $C$ is the degenerate cone $\{0\}\subset \banachb$,  the analysis presented merge in the usual local  convergence analysis on Newton's  method, see \cite{FS09}.

The organization of the paper is as follows.
In Section \ref{sec:int.1},  some notations and  basic results  used in the paper are presented. In Section \ref{lkant}, the main results are stated and  in   Section \ref{sec:PR} some properties of the majorant function are established and the main relationships  between the majorant function and the nonlinear operator used in the paper are presented. In Section \ref{sec:convxk},  the main results are proved and the applications of this results are given in Section \ref{apl}. Some final remarks are made in Section~\ref{fr}.
\subsection{Notation and auxiliary results} \label{sec:int.1}
The following notations and results are used throughout our presentation. We beginning  with the following  elementary convex analysis result:
\begin{proposition} \label{pr:conv.aux1}
Let $I\subset \mathbb{R}$ be an interval and $\varphi:I\to \mathbb{R}$ be convex. For any $s_0\in \mathrm{int}(I)$,   the left derivative there exist (in $\mathbb{R}$)
$$
D^- \varphi(s_0):={\lim}_{s\to s_0 ^-} \; \frac{\varphi(s_0)-\varphi(s)}{s_0-s}
={\sup}_{s<s_0} \;\frac{\varphi(s_0)-\varphi(s)}{s_0-s}. \\
$$
Moreover,  if $s,t,r\in I$, $s<r$, and $s\leqslant t\leqslant r$ then $\varphi(t)-\varphi(s) \leqslant \left[\varphi(r)-\varphi(s)\right][(t-s)/(r-s)].$
\end{proposition}
Let $\banacha$  be a  Banach space. The {\it open} and {\it closed ball} at $x$ with radius $\delta>0$ are denoted, respectively, by
$$
B(x,\delta) := \{ y\in \banacha ~:~ \|x-y\|<\delta \}\;\;\; \mbox{and}\;\;\;B[x,\delta] := \{ y\in \banacha  ~:~\|x-y\|\leqslant \delta
\}.
$$
\begin{proposition} \label{pr:qc}
Let $\{z_k\}$ be a sequence in $\banacha$ and  $\Theta >0$. If  $\{z_k\}$ converges to $z_*$ and  satisfies
\begin{equation} \label{eq:qcfr}
\|z_{k+1}-z_{k}\|\leq \Theta  \|z_{k}-z_{k-1}\|^2, \qquad k=1,2, \ldots.
\end{equation}
then $\{z_k\}$ converges $Q$-quadratically to $z_*$ as follows
$$
\limsup_{k\to \infty } \frac{\|z_{k+1}-z_{*}\|}{\|z_{k}-z_{*}\|^2}\leq \Theta.
$$
\end{proposition}
\begin{proof}
The proof follows the same pattern as the proof of Proposition 2  of \cite{FGO13},    since finite dimensionality plays no role.
\end{proof}
Let $\banacha$ and $\banachb$ be Banach spaces. A  set valued mapping $T: \banacha \rightrightarrows  \banachb $ is called {\it sublinear} or   {\it convex precess}    when its graph is a convex cone, i.e., 
\begin{equation} \label{eq:dsblm}
0\in T(0), \qquad T(\lambda x)=\lambda T(x) \quad  \forall~ \lambda>0, \qquad T(x+x') \supseteq T(x) + T(x'),  ~\quad  \forall ~ x, x'\in \banacha.
\end{equation}
 (sublinear mapping has been extensively studied in \cite{DR2009},  \cite{Rob1972},  \cite{Rocka2} and  \cite{Rocka}) the following definitions and results about  sublinear mappings will be need:  The {\it domain} and  {\it range}   of  a sublinear mapping $T$ are defined, respectively, by
$$
 \mbox{dom\,}T:=\{d\in \banacha~:~Td \neq \varnothing \}, \qquad   \mbox{rge\,}T:=\{y\in \banachb ~:~ y\in T(x) ~\mbox{for some} ~x\in\banacha \}, 
 $$
and the  {\it inverse} $T^{-1}:\banachb \rightrightarrows  \banacha $  of a sublinear mapping $T$  is another sublinear mapping  defined by 
\begin{equation} \label{eq:ding}
 T^{-1}y:=\{d\in \banacha~:~y\in T d\}, \quad  y\in \banachb.
 \end{equation}
The {\it norm} (or inner norm as is called in  see \cite{DR2009})   of  a sublinear mapping $T$ is defined  by
 \begin{equation} \label{eq;dn}
 \|T\|:=\sup \; \{ \|T d\|~: ~d\in  \mbox{dom\,}T, \; \|d\| \leqslant 1 \},
 \end{equation}
where $ \|T d\|:=\inf \{\|v\|~: ~v\in T d  \}$ for  $Td \neq \varnothing $. We use the convention $\|Td \|=+\infty$ for $Td = \varnothing $,   it will be also convenient to use the convention  $Td+ \varnothing=\varnothing$ for all $d\in \banacha$.
\begin{lemma}\label{L:Rocka1}
Let $T: \banacha \rightrightarrows  \banachb $ be a sublinear mapping with closed graph. Then $ \mbox{dom\,}T=\banacha$  if and only if $\|T\|<+\infty$ and $\mbox{rge\,}T=\banachb$ if and only if $ \|T^{-1}\|<+\infty$.
\end{lemma}
\begin{proof}
See Corollary 5C.2 of  \cite{DR2009}.
\end{proof}
Let $S, T: \banacha \rightrightarrows  \banachb $ and $U: \banachb \rightrightarrows  \banachc$  be sublinear mappings.  The   scalar {\it multiplication}, {\it addition}  and {\it composition} of sublinear mappings are sublinear mappings defined, respectively,  by
$$
(\alpha S)(x):=\alpha S(x),  \quad  (S+T)(x):= S(x) + T(x),  ~\quad  UT(x):=\bigcup \left\{ U(y)~: ~ y\in T(x) \right \}, 
$$
 for all $x\in \banacha$ and  $\alpha>0$  and the following norm properties there hold:
\begin{equation} \label{eq;pnor}
\|\alpha S\|=|\alpha| \|S\|, \quad \|S+T\|\leqslant \|S\|+\|T\|, \quad  \|UT\|\leqslant \|U\| \|T\|.
\end{equation}
\begin{remark} \label{r:pn}
Note that definition of the norm in \eqref{eq;dn} implies that  if $ \mbox{dom\,}T=\banacha$ and $A$ is a linear mapping from $\banachc$ to $\banacha$ then 
$\|T(-A)\|=\|TA\|$.
\end{remark}
\begin{lemma}\label{L:Banach}
Let $S, T: \banacha \rightrightarrows  \banachb $ be a sublinear mappings with closed graph such that  $\mbox{dom\,}S= \mbox{dom\,}T=X$ and 
 $
 \|T^{-1}\|<+\infty.
 $
Suppose that $\|T^{-1}\|\|S\|< 1$ and $(S+T)(x)$ is closed for each $x\in \banacha$ then   $\mbox{rge\,}(S+T)=\banacha $ and 
$$
\|(S+T)^{-1}\|\leqslant\frac{\|T^{-1}\|}{1-\|T^{-1}\|\|S\|}.
$$
\end{lemma}
\begin{proof}
These results follows from Theorem~5 of \cite{Rob1972} by taking in account  Lemma~\ref{L:Rocka1}.
\end{proof}
\begin{lemma}\label{LiNg}
Let  $G:[0,1] \to  \banachb$ and  $g:[0,1] \to \mathbb{R}$ be continuous function and $ \banachc$ a reflexive Banach space. Suppose that  $U: \banachb \rightrightarrows  \banachc$ is a sublinear mapping with closed graphic such that $\mbox{dom\,} U \supseteq \mbox{rge\,}G$. If 
$$
\|UG(\tau)\|\leqslant g(\tau), \qquad \tau\in [0,1],
$$
then   there hold
$$
U\int_{0}^{1}G(\tau)d\tau\neq \varnothing, \qquad \qquad \left\|U\int_{0}^{1}G(\tau)d\tau \right\|\leqslant \int_{0}^{1}g(\tau)d\tau.
$$
\end{lemma}
\begin{proof}
see Lemma 2.1 of  \cite{LiNg2012}.
\end{proof}
Let $\Omega\subseteq \banacha$ be an open set and $F:{\Omega}\to \banachb$  a continuously Fr\'echet differentiable function.   The   linear map $F'(x):\banacha \to \banachb$ denotes the  Fr\'echet derivative of $F:{\Omega}\to \banachb$ at $x\in \Omega$.  Let  $C \subset \banachb $ be  a  nonempty closed convex cone,  $z\in\Omega$ and   $T_z: \banacha \rightrightarrows  \banachb $ a  mapping  defined  as
\begin{equation}\label{ro2}
T_{z}d=F'(z)d-C.
\end{equation}
It is well known that the mappings  $T_{z}$  and  $T^{-1}_{z}$ are sublinear  with closed graphic, $ \mbox{dom\,}T_z=X$, $\|T_z\|<+\infty$ and $\mbox{rge\,}T=Y$ if and only if $ \|T^{-1}_{z}\|<+\infty$ (see Lemma~\ref{L:Rocka1} above and  Corollary 4A.7, Corollary 5C.2 and Example 5C.4 of \cite{DR2009} ). Note that 
 \begin{equation}\label{ro3}
 T^{-1}_zy:=\{d\in \banacha~:~ F'(z)d-y\in C\}, \quad  \forall~ z\in  \Omega, ~ \forall~ y\in \banachb.
\end{equation}
\begin{lemma} \label{l:incltr}
Let $\banacha$ and $\banachb$ be  Banach spaces. $\Omega\subseteq \banacha$  an open set and $F:{\Omega}\to \banachb$  a continuously Fr\'echet differentiable function.  Then the following inclusion holds
$$
T_{z}^{-1}F'(v)T_{v}^{-1}w \subseteq T_{z}^{-1}w, \qquad \forall ~ v, z\in  \Omega, ~ \forall~ w\in \banachb. 
$$
As a consequence, 
$$
 \left\|T_{z}^{-1}\left[ F'(y)-F'(x)\right]\right\|\leq \left\|  T_{z}^{-1}F'(v)T_{v}^{-1}\left[ F'(y)-F'(x)\right] \right\|, \qquad \forall ~ v, x, y, z\in  \Omega.
$$
\end{lemma}
\begin{proof}
If $T_{v}^{-1}w=\varnothing$ for some $v \in \Omega$ and  $w\in \banachb$   then the inclusion hods trivially. Assume that  $T_{v}^{-1}w\neq\varnothing$, i.e.,  there exists $d\in T_{v}^{-1}w$ for each $v \in \Omega$ and  $w\in \banachb$.  Note that  definition in  \eqref{ro3} and simples algebraic manipulation gives 
$$
T_{z}^{-1}F'(v)d=\left\{ u\in \banacha~: ~  F'(z)u -w\in C+[ F'(v)d-w] \right\}, \qquad  \forall ~ d\in \banacha, ~ \forall~   w\in \banachb. 
$$
Defition in  \eqref{ro3} implies that,  for each $v \in \Omega$,   $w\in \banachb$ and   $d\in T_{v}^{-1}w$ there holds $F'(v)d-w \in C$.  Thus,  the last equality and   \eqref{ro3} imply that 
$$
T_{z}^{-1}F'(v)d = T_{z}^{-1}w, \qquad  \forall~w\in \banachb, \quad  \forall~ d\in T_{v}^{-1}w, 
$$
which implies the desired inclusion.   To end the proof, note that  the first part of the lemma implies
$$
\left\|T_{z}^{-1}\left[ F'(y)-F'(x)\right]u\right\|\leq \left\|  T_{z}^{-1}F'(v)T_{v}^{-1}\left[ F'(y)-F'(x)\right] u\right\|, \qquad \forall ~ v, x, y, z\in  \Omega,  \quad \forall ~ u\in\banachb. 
$$
Hence,  the inequality of the lemma  follows from the definition of the nom in \eqref{eq;dn}.
\end{proof}
\section{Semi-local analysis  for  Newton's method } \label{lkant}
Our goal is to state and prove a robust semi-local affine invariant theorem for Newton's method to solve nonlinear inclusion  of the form
\begin{equation} \label{eq:Inc}
F(x) \in C, 
\end{equation}
where $F:{\Omega}\to \banachb$ is a  nonlinear continuously   differentiable function,   $\banacha$  and $\banachb$ are Banach spaces and  $\banacha$ is reflexive, $\Omega\subseteq \banacha$ an open set and  $C\subset \banachb $ a nonempty closed convex cone.   For state the theorem we need some definitions. 

  A  nonlinear continuously  Fr\'echet differentiable function $F:{\Omega}\to \banachb$ satisfies the {\it Robson's  Condition} at $\tilde{x}\in \Omega$ if 
\begin{equation} \label{eq:RobCond}
\mbox{rge\,}T_{\tilde{x}}=\banachb, 
\end{equation}
where $T_{\tilde{x}}: \banacha \rightrightarrows  \banachb $  is a sublinear   mapping  as defined in \eqref{ro2}.

Let    $\banacha$  and $\banachb$   be a  Banach spaces,  $\Omega\subseteq \banacha$ an open set and  $R>0$ a scalar constant. A  continuously differentiable scalar function $f:[0,R)\to \mathbb{R}$ is a {\it majorant function}   at a point  $\tilde{x}\in \Omega$ for   a   continuously   differentiable function  $F:{\Omega}\to \banachb$     if 
  \begin{equation}\label{eq:MCAI}
   B(\tilde{x},R)\subseteq \Omega, \qquad \qquad  \left\|T_{\tilde{x}}^{-1}\left[F'(y)-F'(x)\right]\right\| \leqslant  f'(\|y-x\|+\|x-\tilde{x}\|)-f'(\|x-\tilde{x}\|),
  \end{equation}
  for all $x,y\in B(\tilde{x},R)$ such that  $\|x-\tilde{x}\|+\|y-x\|< R$ and satisfies the  following conditions:
  \begin{itemize}
  \item[{\bf h1)}]  $f(0)>0$,   $f'(0)=-1$;
  \item[{\bf h2)}]  $f'$ is convex and strictly increasing;
  \item[{\bf h3)}]  $f(t)=0$ for some $t\in (0,R)$.
  \end{itemize}
We also need of the following  condition on the majorant condition $f$  which will be considered to hold
only when explicitly stated
  \begin{itemize}
  \item[{\bf h4)}]  $f(t)<0$ for some $t\in (0,R)$.
  \end{itemize}
  \begin{remark}
  Since $f(0)>0$ and $f$ is continuous then   condition {\bf h4} implies condition {\bf h3}.
  \end{remark}
  \begin{theorem}\label{th:knt1}
Let $\banacha$, $\banachb$ be Banach spaces  and $\banacha$ reflexive, $\Omega\subseteq \banacha$ an open set,    $F:{\Omega}\to \banachb$  a continuously Fr\'echet differentiable function,   $C\subset \banachb $ a nonempty closed convex cone, $R>0$ and $ f:[0,R)\to \mathbb{R}$ a continuously differentiable function.   Suppose that $\tilde{x}\in \Omega$, $F$ satisfies  the Robson's Condition at  $\tilde{x}$,    $f$ is a majorant function for $F$ at  $\tilde{x}$ and
\begin{equation} \label{KH}
   \left \|T_{\tilde{x}}^{-1}(-F(\tilde{x}))\right\|\leqslant f(0)\,.
\end{equation}
  Then $f$ has a smallest zero $t_*\in (0,R)$, the sequences  generated by Newton's Method for solving   the inclusion 
  $
  F(x)\in C
  $ and  the equation  $f(t)=0$,  with starting
  point  $x_0=\tilde{x}$ and $t_0=0$, respectively,
  \begin{align}   \label{ns.KT}
    x_{k+1}\in x_k  +  \mbox{argmin} \left\{ \|d\|: \,  F(x_k)+F'(x_k)d \in C \right\},  \quad    t_{k+1} ={t_k}- \frac{f(t_k)}{f'(t_k)}, \qquad  k=0,1,\ldots\,.
  \end{align}  are well defined, $\{x_k\}$ is contained in $B(\tilde{x},  t_*)$, $\{t_k\}$ is strictly increasing, is contained in   $[0,t_*)$ and  converge  to  $t_*$ and   satisfy  the inequalities
\begin{equation}\label{eq:bd1}
  \|x_{k+1}-x_{k}\|   \leq  t_{k+1}-t_{k} , \qquad \qquad \|x_{k+1}-x_{k}\|\leq   \frac{t_{k+1}-t_{k}}{(t_{k}-t_{k-1})^2} \|x_k-x_{k-1}\|^2,
  \end{equation}
 for all \( k=0, 1, \ldots\, , \) and \( k=1,2, \ldots\, \), respectively. Moreover,    $\{x_k\}$ converge  to  $x_*\in B[\tilde{x}, t_*]$ such that  $F(x_*)\in C,$  
\begin{equation}\label{eq:lc1}
 \|x_*-x_{k}\|   \leq  t_*-t_{k}, \qquad \qquad t_*-t_{k+1}\leq \frac{1}{2}( t_*-t_{k}), \qquad k=0,1, \ldots\,
\end{equation}
and, therefore,    $\{t_{k}\}$  converges $Q$-linearly to $t_*$ and   $\{x_k\}$   converge $R$-linearly to $x_*$. If, additionally, $f$ satisfies {\bf h4} then the following inequalities hold:
\begin{equation}\label{eq:qcct}
\|x_{k+1}-x_{k}\|\leq    \frac{D^-f'(t_*)}{-2f'(t_*)}  \|x_k-x_{k-1}\|^2,  \qquad t_{k+1}-t_{k} \leq \frac{D^-f'(t_*)}{-2f'(t_*)} ({t_k}-t_{k-1})^2, \qquad k=1,2, \ldots\,, 
  \end{equation}
and, as  a consequence,    $\{x_k\}$ and  $\{t_k\}$  converge $Q$-quadratically
 to $x_*$ and $t_*$, respectively, as follow
 \begin{equation}\label{eq:qcs}
\limsup_{k\to \infty}\frac{\|x_{*}-x_{k+1}\|} {\|x_*-x_{k}\|^2}\leq   \frac{D^-f'(t_*)}{-2f'(t_*)}, \qquad  \qquad t_{*}-t_{k+1} \leq \frac{D^-f'(t_*)}{-2f'(t_*)} ({t_*}-t_{k})^2,  \quad k=0,1, \ldots\,.
  \end{equation}
\end{theorem}
  We will  use the above result to  prove a robust semi-local affine invariant theorem for Newton's method for solving nonlinear inclusion  of the form \eqref{eq:Inc}.  The statement of the  theorem~is:
 \begin{theorem}\label{th:knt2}
Let $\banacha$, $\banachb$ be Banach spaces  and $\banacha$ reflexive, $\Omega\subseteq \banacha$ an open set,    $F:{\Omega}\to \banachb$  a continuously Fr\'echet differentiable function,   $C\subset \banachb $ a nonempty closed convex cone, $R>0$ and $ f:[0,R)\to \mathbb{R}$ a continuously differentiable function.   Suppose that $\tilde{x}\in \Omega$, $F$ satisfies  the Robson's Condition at  $\tilde{x}$,    $f$ is a majorant function for $F$ at  $\tilde{x}$  satisfying  {\bf h4} and
 \begin{equation}
    \label{KH.2}
   \left \|T_{\tilde{x}}^{-1}(-F(\tilde{x}))\right\|\leqslant f(0)\,.
  \end{equation}
Define   $\beta:=\sup\{ -f(t) ~:~ t\in[0,R)  \}$. Let $0\leq \rho<  \beta/2$  and   $g:[0,R-\rho)\to \mathbb{R}$, 
\begin{equation} \label{eq:maj2}
  g(t):=\frac{-1}{f'(\rho)}[f(t+\rho)+2\rho].
\end{equation}
Then $g$ has a smallest zero $t_{*,\rho}\in (0,R-\rho)$, the sequences  generated by Newton's Method for solving   the inclusion 
  $
  F(x)\in C
  $ and  the equation  $g(t)=0$,  with starting  point $x_0=\hat{x}$ for any $\hat{x}\in  B(\tilde{x},  \rho)$ and $t_0=0$, respectively,
  \begin{align}   \label{ns.KT2}
    x_{k+1}\in x_k  +  \mbox{argmin} \left\{ \|d\|: \,  F(x_k)+F'(x_k)d \in C \right\},  \quad    t_{k+1} ={t_k}- \frac{g(t_k)}{g'(t_k)}, \qquad  k=0,1,\ldots\,.
  \end{align}  are well defined, $\{x_k\}$ is contained in $B(\tilde{x},  t_{*,\rho})$, $\{t_k\}$ is strictly increasing, is contained in   $[0,t_{*,\rho})$ and  converge  to  $t_{*,\rho}$ and   satisfy  the inequalities
\begin{equation}\label{eq:bdct2}
  \|x_{k+1}-x_{k}\|   \leq  t_{k+1}-t_{k}, \qquad k=0, 1, \ldots\, , 
  \end{equation}
  \begin{equation}\label{eq:qbdct2}
 \|x_{k+1}-x_{k}\|\leq   \frac{t_{k+1}-t_{k}}{(t_{k}-t_{k-1})^2} \|x_k-x_{k-1}\|^2\leq  \frac{D^-g'(t_{*,\rho})}{-2g'(t_{*,\rho})}  \|x_k-x_{k-1}\|^2,  \qquad k=1, 2, \ldots\,
  \end{equation}
   Moreover,    $\{x_k\}$ converge  to  $x_*\in B[\tilde{x}, t_{*,\rho}]$ such that  $F(x_*)\in C$,  satisfies the inequalities
\begin{equation}\label{eq:lcr2}
 \|x_*-x_{k}\|   \leq  t_{*,\rho}-t_{k}, \qquad \qquad t_{*,\rho}-t_{k+1}\leq \frac{1}{2}( t_{*,\rho}-t_{k}), \qquad k=0,1, \ldots\,
\end{equation}
and the convergence of $\{x_k\}$ and  $\{t_k\}$ to  $x_*$ and $t_{*,\rho}$, respectively,  are  $Q$-quadratic as follows
 \begin{equation}\label{eq:qcrt2}
\limsup_{k\to \infty}\frac{\|x_{*}-x_{k+1}\|} {\|x_*-x_{k}\|^2}\leq   \frac{D^-g'(t_{*,\rho})}{-2g'(t_{*,\rho})},   \qquad t_{*,\rho}-t_{k+1} \leq \frac{D^-g'(t_{*,\rho})}{-2g'(t_{*,\rho})} ({t_{*,\rho}}-t_{k})^2,  \quad k=0,1, \ldots\,.
  \end{equation}
\end{theorem}
\begin{remark}
Using Lemma~\ref{L:Rocka1} its easy to see that the inequalities  \eqref{eq:MCAI},  \eqref{KH}  and \eqref{KH.2}  are well defined.  
\end{remark}
\begin{remark} \label{re:edef}
It follows from \eqref{ro2} and \eqref{ro3} that definition of the sequence $\{x_k\}$ in  \eqref{ns.KT} is equivalent to  the conditions
$$
x_{k+1}-x_k \in T_{x_k}^{-1}(-F(x_k)) \qquad \mbox{and} \qquad \left\|x_{k+1}-x_k \right\|=\left\|T_{x_k}^{-1}(-F(x_k))\right\|,  \qquad k=0,1, \ldots .
$$
\end{remark}
\begin{remark} \label{re:afilc}
Theorems~\ref{th:knt1} and \ref{th:knt2}  are  affine-invariant in the following sense: Letting  $A:  \banachb \to  \banachb$ be an invertible continuous linear mapping,     $\tilde{F} := A \circ F$ and the set $\tilde{C} := A(C)$, the corresponding inclusion problem \eqref{eq:Inc} is given by
$$
\tilde{F}(x) \in \tilde{C}, 
$$
and the convex process associated   is  denoted by $\tilde{T}_{\tilde{x}}d := \tilde{F} (\tilde{x})d - \tilde{C}$. Then $\tilde{T}_{\tilde{x}}= A \circ T_{\tilde{x}}$ and $\tilde{T}^{-1}_{\tilde{x}}= T^{-1}_{\tilde{x}} \circ A^{-1}$. Moreover, the conditions  $\mbox{rge\,}\tilde{T}_{\tilde{x}}=Y$, 
$$
 \left \|\tilde{T}_{\tilde{x}}^{-1}(- \tilde{F}(\tilde{x}))\right\|\leqslant f(0), 
$$
and  the affine majorant  condition (Lipschitz-like condition)
$$
\left\|\tilde{T}_{\tilde{x}}^{-1}\left[|\tilde{F}'(y)-|\tilde{F}'(x)\right]\right\| \leqslant  f'(\|y-x\|+\|x-\tilde{x}\|)-f'(\|x-\tilde{x}\|), 
$$  
for $x,y\in B(\tilde{x},R)$,  $\|x-\tilde{x}\|+\|y-x\|< R$, are equivalent to $\mbox{rge\,}T_{\tilde{x}}=\banachb$, \eqref{KH}  and   \eqref{eq:MCAI} respectively.   Therefore, the assumptions in Theorems~\ref{th:knt1} and \ref{th:knt2}  are insensitive with respect to invertible continuous linear mappings.  Note that the results of \cite{Rob1972-2}  do not have such property. For more details about  affine invariant theorems  see  \cite{DHB} and \cite{DH}.
\end{remark}
\subsection{Preliminary results} \label{sec:PR}
In this section, we will prove all the statements in Theorems~\ref{th:knt1} and \ref{th:knt2} regarding to the majorant function and  the sequence $\{t_k\}$ associated. The main relationships between the majorant function and the nonlinear operator  will be also established.

\subsection{The majorant function}
In this section we will study the majorant function $f$ and prove all
results regarding only the sequence $\{t_k\}$.  Define
\begin{equation}
  \label{eq:def.bart}
  \bar{t}:=\sup \left\{t\in [0,R): f'(t)<0 \right\}\;.
\end{equation}
\begin{proposition} \label{pr:maj.f}
  The majorant function $f$  has a smallest root $t_*\in  (0,R)$, is strictly convex and
\begin{equation}  \label{eq:n.f}
   f(t)>0, \quad f'(t)<0, \qquad t<t-f(t)/f'(t)< t_*,  \qquad\qquad  \forall \,t\in [0,t^*) .
\end{equation}
Moreover, $f'(t_*)\leqslant 0$ and
  \begin{equation}  \label{eq:pr.1.b}
     f'(t_*)<0\iff \exists\, t\in (t_*,R), \;f(t)< 0 .
  \end{equation}
  If, additionally, $f$ satisfies condition  {\bf h4} then the following statements  hold:
   \begin{itemize}
  \item[i)] $f'(t)<0$ for any $t\in [0,\bar t)$;
  \item[ii)] $0< t_* < \bar t\leq R$;
  \item[iii)] $ \beta=-\lim_{t\to \bar t_{-}} f(t),\qquad 0< \beta <\bar t$;
    \item[iv)]  If  $0\leq \rho<\beta/2$  then $\rho<\bar t/2 <\bar t$ and  $f'(\rho)<0$.
\end{itemize}
\end{proposition}
\begin{proof}
See    Propositions 2.3 and 5.2 of \cite{FS12} and Proposition 3 of \cite{FS09}.
\end{proof}
In view of the second inequality in (\ref{eq:n.f}), Newton iteration is  well defined in $[0,t_*)$. Let us call it
\begin{equation} \label{eq:n.f.2}
  \begin{array}{rcl}
  n_f:[0,t_*)&\to& \mathbb{R}\\
    t&\mapsto& t-f(t)/f'(t).
  \end{array}
\end{equation}
\begin{proposition} \label{pr:2}
  Newton iteration  $n_f$ is strictly increasing,    maps $[0,t^*)$ in $[0,t^*)$.
\end{proposition}
\begin{proof}
See    Proposition 4 of \cite{FS09}.
\end{proof}
The definition of $\{t_k\}$ in Theorem~\ref{th:knt1} is equivalent to the following one
\begin{equation}\label{eq:tknk}
  t_0=0,\quad t_{k+1}=n_f(t_k), \qquad k=0,1,\ldots\, .
\end{equation}
\begin{corollary} \label{cr:kanttk}
 The sequence $\{t_{k}\}$ is well defined, is strictly increasing, is contained in $[0,t_*)$ and converges $Q$-linearly to $t_*$  as second inequality in \eqref{eq:lc1}. If $f$ also satisfies  {\bf h4}, then   the second inequality in \eqref{eq:qcct} holds and, moreover,  $\{t_{k}\}$ converges  $Q$-quadratically to $t_*$  as  in \eqref{eq:qcs}.
\end{corollary}
\begin{proof}
 The proof follows the same pattern as the proof of  Corollary 2.15  of \cite{FGO13}.
\end{proof}
\subsection{Relationship between the majorant function and the nonlinear operator} \label{sec:rbmno}
In this section, we will present the main relationships between the majorant function $f$ and the nonlinear operator $F$ that we need for proving Theorem~\ref{th:knt1}.
\begin{proposition} \label{wdns}
  If \,\,$\| x-\tilde{x}\|\leqslant t< \bar{t}$ then   $\mbox{dom\,}(T_{x}^{-1}F'(\tilde{x}))=\banacha$ and there holds
$$
\left\|T_{x}^{-1}F'(\tilde{x})\right\|\leqslant  -1/f'(t). 
$$
As a consequence, $\mbox{rge\,}T_{x}=\banachb$.
\end{proposition}
\begin{proof}
Take $0\leqslant t<\bar{t}$ and $x\in B[\tilde{x},t]$.  Since $f$ is a majorant function for $F$ at  $\tilde{x}$,  using the   \eqref{eq:MCAI}, {\bf h2}, {\bf h1} and  Proposition~\ref{pr:maj.f} item i we  obtain
\begin{align*}
  \left\|T_{\tilde{x}}^{-1}[F'(x)-F'(\tilde{x})]\right\|
    &\leqslant  f'(\|x-\tilde{x}\|)-f'(0)\\
    &\leqslant   f'(t)-f'(0)\\
    &=f'(t)+1<1.
  \end{align*}
  For simplify the notation define $S=T_{\tilde{x}}^{-1}[F'(x)-F'(\tilde{x})]$. 
Since $[F'(x)-F'(\tilde{x})]$ is a continuous linear mapping  and $T_{\tilde{x}}^{-1}$ is   sublinear mapping with closed graph,  it easy to see that    $S$ is a  sublinear mapping with closed graph. Moreover,   as by assumption  $\mbox{rge\,}T_{\tilde{x}}=\banachb$ we have  $\mbox{dom\,}S=\banacha$.  Because $S$ has closed graph,  it easy to see that $(S+I)(x)$  has also  closed graph   for all $x\in \banacha$,  where  $I$  is the identity mapping  on $\banacha$. Therefore,  applying  Lemma~\ref{L:Banach}  with $T=I$ and taking in account  above inequality we conclude that  $\mbox{rge\,}(T_{\tilde{x}}^{-1}[F'(x)-F'(\tilde{x})]  + I)=\banacha$  and
  \begin{equation}  \label{eq:bn1}
\left \|\left(T_{\tilde{x}}^{-1}[F'(x)-F'(\tilde{x})]  + I\right)^{-1}\right\| \leqslant \frac{1}{1-\left(f'(t)+1\right)}= \frac{1}{-f'(t)}.
  \end{equation}
 Since $0\in C$  we have   $T_{\tilde{x}}^{-1}F'(\tilde{x})d\ni d$ for all $d\in \banacha$. Thus, as    $T_{\tilde{x}}^{-1}$ is   sublinear mapping, using  the last inclusion in \eqref{eq:dsblm} we obtain that 
 $$
 \left(T_{\tilde{x}}^{-1}F'(x)\right)d\supseteq   \left(T_{\tilde{x}}^{-1}[F'(x)-F'(\tilde{x})]  + I\right)d, \qquad  \forall ~ d\in \banacha.
 $$
 In particular, as  we know that $\mbox{rge\,}(T_{\tilde{x}}^{-1}[F'(x)-F'(\tilde{x})]  + I)=\banacha$,  last inclusion  implies that 
\begin{equation} \label{eq:rgebt1}
 \mbox{rge\,}\left(T_{\tilde{x}}^{-1}F'(x)\right)=\banacha.
\end{equation}
Hence, utilizing again  above inclusion and definition of the inverse in \eqref{ro3} we easily  conclude that  
$$
 \left(T_{\tilde{x}}^{-1}F'(x)\right)^{-1}v\supseteq   \left(T_{\tilde{x}}^{-1}[F'(x)-F'(\tilde{x})]  + I\right)^{-1}v, \qquad \forall~ v\in \banachb.
$$
Taking into account  the definition of the norm in \eqref{eq;dn}, the last inclusion implies that   
  \begin{equation} \label{eq:bn2}
\left\|\left(T_{\tilde{x}}^{-1}F'(x)\right)^{-1} \right\| \leqslant  \left \| \left(T_{\tilde{x}}^{-1}[F'(x)-F'(\tilde{x})]  + I\right)^{-1}\right\|.
  \end{equation}
 On the other hand, as $F'(x)$  is a linear mapping,  using the definitions in  \eqref{eq:ding} and  \eqref{ro3} we obtain that
 \begin{align*}
 - \left(T_{\tilde{x}}^{-1}F'(x)\right)^{-1}d &=\{-v\in \banacha~:~ F'(\tilde{x})d-F'(x)v\in C\} \notag\\
                                                               & = \{u\in \banacha~:~  F'(x)u+F'(\tilde{x})d\in C\}\\
                                                               &= T_{x}^{-1}(-F'(\tilde{x}))d, \qquad \notag
\end{align*}
for all  $d\in \banacha$. Thus using the  last equality, definition of the norm in \eqref{eq;dn}, \eqref{eq:rgebt1} and Lemma~\ref{L:Rocka1} we have 
\begin{equation}  \label{eq:bn3}
\left\| T_{x}^{-1}(-F'(\tilde{x}))  \right\| = \left\|  - \left(T_{\tilde{x}}^{-1}F'(x)\right)^{-1}\right\|= \left\| \left(T_{\tilde{x}}^{-1}F'(x)\right)^{-1}\right\|<+\infty.
\end{equation}
Thus, since  $F'(\tilde{x})$  is a linear mapping,  Remark~\ref{r:pn} and last inequaity allow us to conclude  that
$$
\left\| T_{x}^{-1}F'(\tilde{x}) \right\|=\left\| T_{x}^{-1}(-F'(\tilde{x}))  \right\|<+\infty, 
$$
which combined with Lemma~\ref{L:Rocka1} implies  the first statement of the proposition. Moreover, the desired  inequality follows by combination of  \eqref{eq:bn1}, \eqref{eq:bn2},  \eqref{eq:bn3} with the last equality. 

For proof the last statement  of the proposition, first note that   definition of the norm in \eqref{eq;dn} and Lemma~\ref{l:incltr}  give us
$$
\| T_{x}^{-1}\| \leqslant \|T_{x}^{-1}F'(\tilde{x})T_{\tilde{x}}^{-1}\|.
$$
From Lemma~\ref{L:Rocka1}  the assumption  $\mbox{rge\,}T_{\tilde{x}}=Y$  implies $\|T_{\tilde{x}}^{-1}\|<+ \infty$ and first part of the proposition implies $\|T_{x}^{-1}F'(\tilde{x})\|<\infty$, hence \eqref{eq;pnor}  and above inequality yields $\| T_{x}^{-1}\|<+ \infty$ which, by using again Lemma~\ref{L:Rocka1},  gives the desired result.
\end{proof}
Newton's iteration  at a point $x\in \Omega$ happens to be a solution  of the linearization
of the inclusion $F(y)\in C$ at such a point, namely, a solution of the linear  inclusion $F(x)+F'(x)(x-y)\in C$ .  So, we study the  linearization error of $F$   at a point
in $\Omega$
\begin{equation}\label{eq:def.er}
  E(x,y):= F(y)-\left[ F(x)+F'(x)(y-x)\right],\qquad y,\, x\in \Omega.
\end{equation}
We will bound this error by the error in the linearization on the
majorant function $f$
\begin{equation}\label{eq:def.ers}
        e(t,s):=f(s)-[f(t) +f'(t)(s-t)],   \qquad
     t, s\in [0,R).
\end{equation}
\begin{lemma}  \label{pr:taylor}
Take
$
x,y\in B(\tilde{x},R) \quad\mbox{and}\quad 0\leqslant t<s<R.
$
If $\|x-\tilde{x}\|\leqslant t$ and $\|y-x\|\leqslant s-t$, then 
\[ 
\left\|T_{\tilde{x}}^{-1}(-E(x,y))\right\|\leqslant e(t, s)\left[\frac{\|y-x\|}{s-t}\right]^2. 
\]
\end{lemma}
\begin{proof}
  As  $x, y\in B(\tilde{x},R)$  and the ball is convex
  $
  x+\tau(y-x)\in B(\tilde{x},R),
  $
  for all $\tau\in [0,1]$.   Since, by assumption,  $\mbox{rge\,}T_{\tilde{x}}=\banachb$ we obtain   that $\mbox{dom\,}T_{\tilde{x}}^{-1}=\banachb$. On the other hand,   taking into account  Remark~\ref{r:pn}  and that   $F'(.)$  is a linear mapping on $ \mbox{dom}F$,  we conclude
$$
 \left\|T_{\tilde{x}}^{-1}\left(-[F'(x+\tau(y-x))-F'(x)](y-x)\right) \right\| \leq \left\|T_{\tilde{x}}^{-1}[F'(x+\tau(y-x))-F'(x)]\right\|\left\|(y-x) \right\|, 
$$
for all $\tau\in [0,1]$. Hence, as $f$ is a majorant function for $F$ at  $\tilde{x}$,  using the   \eqref{eq:MCAI} and last inequality we have 
$$
  \left\|T_{\tilde{x}}^{-1}\left(-[F'(x+\tau(y-x))-F'(x)](y-x)\right) \right\|\leqslant \left[f'\left(\|x-\tilde{x}\|+\tau\left\|y-x\right\|\right)-f'\left(\|x-\tilde{x}\|\right)\right]\|y-x\|,
  $$
for all $\tau\in [0,1]$. Thus, as $\mbox{dom\,}T_{\tilde{x}}^{-1}=\banachb$  we apply   Lemma~\ref{LiNg} with  $U=T_{\tilde{x}}^{-1}$,   $G(\tau)$ equal to the expression  in  parentheses on  the left had side of last inequality and $g(\tau)$ equal to the expression on the right hand site of that inequality,  obtaining   
\begin{multline}
    \left\|  T_{\tilde{x}}^{-1}\int_0 ^1 -[F'(x+\tau(y-x))-F'(x)](y-x)\; d\tau \right\| \\ \leqslant \int_0 ^1 
    \left[f'\left(\|x-\tilde{x}\|+\tau\left\|y-x\right\|\right)-f'\left(\|x-\tilde{x}\|\right)
    \right]
    \|y-x\|\;d\tau\,, 
\end{multline}
Using the convexity of $f'$, that $\|x-\tilde{x}\|<t$,
  $\|y-x\|<s-t$, $s<R$ and Proposition~\ref{pr:conv.aux1} we have
  \begin{align*}
    f'\left(\|x-\tilde{x}\|+\tau\|y-x\|\right)-f'\left(\|x-\tilde{x}\|\right)
    &\leqslant f'\left(t+\tau\|y-x\|\right)-f'\left(t\right)\\
    &\leqslant [f'(t+\tau(s-t))-f'(t) ] \frac{\|y-x\|}{s-t}\,,
  \end{align*}
  for any $\tau\in [0,1]$. Combining the two above inequalities  we obtain
  \[ 
		 \left\|  T_{\tilde{x}}^{-1}\int_0 ^1 -[F'(x+\tau(y-x))-F'(x)](y-x)\; d\tau \right\| \leqslant \int_0 ^1 [f'(t+\tau(s-t))-f'(t) ] \frac{\|y-x\|^2}{s-t}\, d\tau,
  \]
  which, after performing the integration of the right hand side,  taking in account the definition of  $e(t, s)$ in \eqref{eq:def.ers}  and that \eqref{eq:def.er} is equivalent to 
   \[
   E(x,y)=\int_0 ^1 [F'(x+\tau(y-x))-F'(x)](y-x)\; d\tau,
   \]
  yields the desired result.
  \end{proof}
Since  $ \banacha$ is reflexive, second part of Lemma \ref{wdns} guarantees, in particular,  that   the {\it Newton's step set $D_{F,C}(x)$ at $x\in B(\tilde{x},t_*)$ associated to $F$ and $C$}  is nonempty,  that  is, 
\begin{align} \label{eq:nstep}
D_{F,C}(x)&:= \mbox{argmin} \left\{ \|d\| ~: ~  F(x)+F'(x)d \in C \right\}\\   \notag
                   &=  \mbox{argmin} \left\{ \|d\| ~: ~  d\in T_{x}^{-1}(-F(x)) \right\}  \\  
                   &\neq \varnothing, \notag
\end{align}
 and consequently, the Newton's iteration multifunction is well defined in $ B(\tilde{x},t_*)$.
  Let us call  $N_{F, C}$   the {\it Newton's iteration multifunction for $F$ and $C$} defined  in $ B(\tilde{x},t_*)$:
\begin{equation} \label{NF}
  \begin{array}{rcl}
  N_{F,C}: B(\tilde{x},t_*)&  \rightrightarrows & \banacha\\
    x&\mapsto& x+D_{F,C}(x).
  \end{array}
\end{equation}
One can apply a \emph{single} Newton's iteration multifunction on any $x\in  B(\tilde{x},t_*)$ to obtain the set $N_{F, C}(x)$
which may not is contained 
to $ B(\tilde{x},t_*)$, or even may not in the domain of $F$.
So, this is enough to guarantee the 
well-definedness of only one iteration.   To ensure that Newtonian
iteration multifunction  may be repeated indefinitely, we need some
additional results. 

First, define some subsets of $B(\tilde{x}, t_*)$ in which, as we shall
prove, Newton iteration \eqref{NF} is ``well behaved''.
\begin{align}\label{E:K}
K(t)&:=\left\{ x\in   \banacha ~: ~ \|x-\tilde{x}\|\leqslant t\, ,  \;
   \left\|T_{x}^{-1}(-F(x))\right\| \leqslant -\frac{f(t)}{f'(t)}\right\},\qquad
   t\in [0,t_*)\,,\\
  \label{eq:def.K}
 K&:=\bigcup_{t\in[0,t_*)} K(t).
\end{align}
In \eqref{E:K}, $0\leqslant t<t_*$, therefore, $f'(t)\neq 0$ and $\mbox{rge\,}T_{x}=\banachb$
 in $B[\tilde{x},t]\subset B[\tilde{x},t_*)$  (Proposition~\ref{wdns}).
So, the above definitions are consistent.
\begin{lemma} \label{l:iset1}
For each $t\in [0, t_*)$, $x\in K(t)$ and   $x_{+}\in N_{F,C}(x)$ there hold:
\begin{itemize}
 \item[{\bf i)}]  $ \|x_{+}-x\| \leq n_f(t)-t$;
 \item[{\bf ii)}]  $\|x_{+}-\tilde{x}\| \leq n_{f}(t)<t_*$;
 \item[{\bf iii)}]  $\displaystyle \|z-x_{+}\|\leq  -\frac{f(n_f(t))}{f'(n_f(t))}\left[ \frac{\|x_{+}-x\|}{n_f(t)-t}\right]^2,$ for all $z\in N_{F,C}(x_{+}).$
\end{itemize}
\end{lemma}
\begin{proof}
Take $t\in[0,t_*)$, $x\in K(t)$.  Using definition in  \eqref{E:K} and the  statements in Proposition~\ref{pr:2} we have
\begin{equation} \label{eq:eq.aux.k}
   \| x-\tilde{x}\|\leqslant t,\qquad \|T_{x}^{-1}(-F(x))\| \leqslant-f(t)/f'(t),\quad t<n_f(t)<t_*.
\end{equation}
Take    $x_{+}\in N_{F,C}(x)$. From the definition in   \eqref{NF}  we have   $x_{+}-x \in D_{F,C}(x)$,  which   taking  into account definition in \eqref{eq:nstep} and  second inequality  in  \eqref{eq:eq.aux.k}  yields 
\begin{equation*} \label{eq;cnes}
\|x_{+}-x\|=\| T_{x}^{-1}(-F(x))\| \leqslant-f(t)/f'(t),
\end{equation*}
and by using definition of $n_f$ in \eqref{eq:n.f.2} we obtain the item {\bf i}. 

Combining  first and last inequalities  in \eqref{eq:eq.aux.k},  item {\bf i} and  definition of $n_f$ in  \eqref{eq:n.f.2}   we have
\begin{align*}
\|x_{+}-\tilde{x}\| \leqslant \|x-\tilde{x}\|+\|x_{+}-x\| \leqslant t-f(t)/f'(t)=n_f(t)<t_*\,,
\end{align*}
which is the inequality in item {\bf ii}.

Now we are going to prove the last item of the lemma. Form item {\bf ii}   we conclude that
\begin{equation} \label{eq:auxNfNF}
 N_{F,C}(x)\subseteq B[\tilde{x}, n_f(t)]\subset B(\tilde{x}, t_*), 
\end{equation}
which, in particular,  implies that the set  $N_{F,C}(x)$ is contained in  the domain of   the function $F$.  Therefore,  we claim that the following  relations   hold:
\begin{equation} \label{eq:feq}
\varnothing \neq T_{x_{+}}^{-1}F'(\tilde{x})T_{\tilde{x}}^{-1}(-E(x,x_{+})) \subset T_{x_{+}}^{-1}(-F'(x_{+})), \qquad \forall ~ x_{+}\in N_{F,C}(x).
\end{equation}
where $E$ is the  linearization error of $F$  as defined in \eqref{eq:def.er}.    Assumption,  $\mbox{rge\,}T_{\tilde{x}}=\banachb$ implies   that $\mbox{dom\,}T_{\tilde{x}}^{-1}=\banachb$. Thus, as $ \|x_{+}-\tilde{x}\| <  t_*\leq \bar{t}$  using that  $\mbox{dom\,}T_{\tilde{x}}^{-1}=\banachb$ together with Proposition~\ref{wdns}  the first claim in \eqref{eq:feq} follows.  For proving   the inclusion  in \eqref{eq:feq},  we first use Lemma~\ref{l:incltr} for concluding that
$$
T_{x_{+}}^{-1}F'(\tilde{x})T_{\tilde{x}}^{-1}(-E(x,x_{+}))\subset T_{x_{+}}^{-1}(-E(x,x_{+})), \qquad \forall ~ x_{+}\in N_{F,C}(x).
$$
Since $x_{+}\in N_{F,C}(x)$  it follows from \eqref{eq:nstep} and \eqref{NF}  that  $ F(x)+F'(x)(x_{+}-x) \in C$. Hence, using definition in \eqref{ro3} we easy conclude that 
$$
T_{x_{+}}^{-1}(-E(x,x_{+}))= T_{x_{+}}^{-1}(-F'(x_{+})), \qquad \forall ~ x_{+}\in N_{F,C}(x).
$$
Therefore, combining two last inclusion we conclude that  the   inclusion \eqref{eq:feq} holds, as claimed.   On the other hand, since  $n_f(t)$ belongs to the domain of  $f$, using the definitions of Newton iterations on \eqref{eq:n.f.2} and definition of the linearization error  \eqref{eq:def.ers},  we obtain
\begin{equation} \label{eq:nderror}
 f(n_f(t))=f(n_f(t))-\left[f(t)+f'(t)(n_f(t)-t)\right]=e(t,n_f(t)).
\end{equation}
From \eqref{eq:eq.aux.k} we have  $\| x-\tilde{x}\|\leqslant t$ and $t<n_f(t)<t_*<R$. Hence,    the  first inequality in \eqref{eq:eq.aux.k} and item {\bf i} allow us to apply   Lemma~\ref{pr:taylor}
with $y=x_{+}$ and $s=n_{f}(t)$ to have
\begin{equation*}
     \| T_{\tilde{x}}^{-1}(-E(x,x_{+})) \| \leqslant e(t,n_f(t))\left[\frac{\|x_{+}-x\|}{f(t)/f'(t)}\right]^2=f(n_f(t)) \left[\frac{\|x_{+}-x\|}{f(t)/f'(t)}\right]^2,
\end{equation*}
where, in above expression,  was  used \ \eqref{eq:nderror} to obtain the equality.
As  $x_{+}\in N_{F,C}(x)$  implies that $\|x_{+}-\tilde{x}\|\leqslant n_f(t)<t_*\leq \bar{t}$, it follows from Proposition \ref{wdns} that $\mbox{dom\,}(T_{x_{+}}^{-1}F'(\tilde{x}))=\banacha$  and
\[ \|T_{x_{+}}^{-1}F'(\tilde{x})\|\leqslant -1/f'(n_f(t)).\]
Combining   \eqref{eq:feq}  with the two above inequalities and  by using property of the norm    we obtain  
\begin{align*}
\left\| T_{x_{+}}^{-1}(-F'(x_{+}))\right\| &\leq \left\| T_{x_{+}}^{-1}F'(\tilde{x})T_{\tilde{x}}^{-1}(-E(x,x_{+}))\right\| \\
                                                                         &\leqslant \left\| T_{x_{+}}^{-1}F'(\tilde{x})\right\| \left\|T_{\tilde{x}}^{-1}(-E(x,x_{+}))\right\| \\
                               &\leqslant -\frac{f(n_f(t))}{f'(n_f(t))} \left[\frac{\|x_{+}-x\|}{f(t)/f'(t)}\right]^2,
\end{align*}
hence,  as $\|z-x_{+}\|=\left\| T_{x_{+}}^{-1}(-F'(x_{+}))\right\| $ for all $z\in N_{F,C}(x_{+})$,  the item {\bf iii} is proved.
\end{proof}

\begin{lemma} \label{NfNF}
For each $t\in [0, t_*)$ the following inclusions  hold: $K(t)\subset B(\tilde{x},t_*)$ and
$$N_{F,C}\left( K(t) \right)\subset K\left( n_f(t) \right).$$
As a consequence,
$K\subset B(\tilde{x},t_*)$ and $N_{F,C}(K)\subset K.$
\end{lemma}
\begin{proof}
The first inclusion follows trivially from the definition of $K(t)$. Combining items {\bf i} and {\bf iii} of Lemma~\ref{l:iset1} and taking in account  that 
 $$
 \|z-\tilde{x}\|=\left\| T_{\tilde{x}}^{-1}(-F'(\tilde{x}))\right\|, 
 $$
for all $z\in N_{F,C}(\tilde{x})$,  it easily to conclude that the following inequality holds
$$
\left\| T_{\tilde{x}}^{-1}(-F'(\tilde{x}))\right\|   \leq  -\frac{f(n_f(t))}{f'(n_f(t))}.
$$
This inequality together with   item {\bf ii} of Lemma~\ref{l:iset1}  proof the second inclusion. The next inclusion (first on the second sentence), follows trivially
from definitions \eqref{E:K} and \eqref{eq:def.K}.  To verify the last
inclusion, take $x\in K$.  Then $x\in K(t)$ for some $t\in [0,t_*)$.
Using the  first part of the lemma, we conclude that $N_{F,C}(x)\subseteq
K(n_f(t))$. To end the proof, note that $n_f(t)\in [0,t_*)$ and use
the definition of $K$.
\end{proof}
\subsection{Convergence}\label{sec:convxk}
Finally, we are ready to prove the main results of the paper, namely, the Theorem~\ref{th:knt1} and Theorem~\ref{th:knt2}. First we note that, by using \eqref{eq:nstep} and \eqref{NF},  the
sequence $\{x_k\}$ (see \eqref{ns.KT}) satisfies 
\begin{equation} \label{NFS}
 x_{k+1}\in N_{F,C}(x_k),\qquad k=0,1,\ldots \,,
\end{equation}
which is indeed an equivalent definition of this sequence.
\begin{proof}[Proof of Theorem~\ref{th:knt1}]
All statements  involving only  $\{t_k\}$ were proved in Corollary~\ref{cr:kanttk}, namely,  $\{t_k\}$ is strictly increasing, is contained in   $[0,t_*)$,  converges $Q$-linearly to $t_*$ as in second inequality in \eqref{eq:lc1} and if $f$ satisfies  {\bf h4} then   $\{t_k\}$  satisfies  the second inequality in \eqref{eq:qcct} and converge $Q$-quadratically as the  second inequality in   \eqref{eq:qcs}. 

Now we are going to prove the statements  involving the sequence $\{x_k\}$ with starting point $x_{0}=\tilde{x}$.   From \eqref{KH} and  {\bf h1},  we have
\[ 
x_{0}=\tilde{x} \in K(0)\subset K,
\]
where the second inclusion follows trivially from  \eqref{eq:def.K}.  Using the above inclusion,  
the inclusions $N_{F,C}(K)\subset K$ in Lemma \ref{NfNF} and \eqref{NFS} we conclude that the sequence
$\{x_k\}$ is well defined and rests in $K$.  From the first inclusion on second part of the Lemma \ref{NfNF}
we have trivially that $\{x_k\}$ is contained in $B(\tilde{x}, t_*)$.

 We will prove, by induction that
\begin{equation} \label{eq:xktk}
        x_k\in K(t_k), \qquad k=0,1,\ldots \,.
\end{equation}
The above inclusion, for  $k=0$ follows from  \eqref{KH}, assumption   {\bf h1} and definition of $K(0)$ in \eqref{E:K}. 
Assume now that $x_k\in K(t_k)$. Thus, using  Lemma \ref{NfNF}, \eqref{NFS} and \eqref{eq:tknk}  we conclude that $x_{k+1}\in K(t_{k+1}),$ which completes the induction  proof of \eqref{eq:xktk}.

Now, combining \eqref{eq:tknk}, \eqref{NFS}, \eqref{eq:xktk} and item~{\bf i} of Lemma~\ref{l:iset1} we obtain first inequality in    \eqref{eq:bd1}. Since $\{t_k\}$ converges to $t_*$, the first  inequality \eqref{eq:bd1} implies
\[ 
   \sum_{k=k_0}^\infty \|x_{k+1}-x_k\|\leqslant
   \sum_{k=k_0}^\infty t_{k+1}-t_k =t_*-t_{k_0}<+\infty,
\]
for any $k_0\in\mathbb{N}$. Hence, $\{x_k\}$ is a Cauchy sequence in
$B(\tilde{x}, t_*)$ and so, converges to some $x_*\in B[\tilde{x},t_*]$. 
The above inequality also implies  that $\|x_*-x_k\|\leqslant
t_*-t_k$, for any $k$. For  proving  that $F(x_*)\in C$. First, observe that the first  inequality in \eqref{eq:lc1}  implies that  $\|x_*-\tilde{x}\|\leq t_*<R$. Hence, as $B(\tilde{x},R)\subseteq \Omega$ we conclude that  $x_*\in \Omega$. Since definition of $\{x_k\}$ in \eqref{ns.KT} implies that
$$
 F(x_k)+F'(x_k)(x_{k+1}-x_k) \in C,\qquad k=0,1,\ldots\,, 
$$
and $F$ is continuous differentiable in $\Omega$, $\{x_k\}\subset B(\tilde{x},t_*)  \subset B(\tilde{x},R)$, $\{x_k\}$  converges to $x_*\in \Omega$ and  $C\subset \banachb $  is closed, the result follows by taking limit as $k$ goes to infinite  in   above inclusion.

In order to prove  the second inequality in \eqref{eq:bd1}, first note that $  x_k\in K(t_{k})$, $x_{k+1}\in N_{F,C}(x_k)$  and $t_{k+1}=n_f(t_k)$, for all $k=0,1,\ldots .$
Thus,  apply item~{\bf iii} of  Lemma~\ref{l:iset1}  with $x=x_{k-1}$, $z=x_{k+1}$, $x_{+}=x_{k}$  and $t=t_{k-1}$
to obtain
$$
\|x_{k+1}-x_k\|\leq -\frac{f(t_k)}{f'(t_k)}\left[ \frac{\|x_k-x_{k-1}\|}{t_{k}-t_{k-1}}\right]^2,
$$
which, using second inequality in \eqref{ns.KT}  yields the desired  inequality.

Now,  we assume that {\bf h4} holds. Therefore, combining both the second inequalities in \eqref{eq:qcct} and \eqref{eq:bd1}, we obtain the first inequality in \eqref{eq:qcct}. To establish the first inequality in  \eqref{eq:qcs}, use the  first inequality in \eqref{eq:qcct} and  Proposition~\ref{pr:qc} with $z_k=x_k$ and $\Theta=D^-f'(t_*)/(-2 f'(t_*))$. Thus, the proof is concluded.
\end{proof}
Now we are going to prove Theorem~\ref{th:knt2}, but first  we need   two more additional results. 
\begin{proposition}  \label{pr:ar1}
For any $y\in B(\tilde{x},R)$, the following inequality holds
$$
 \left \|T_{\tilde{x}}^{-1}(-F(y))\right\|  \leq f(\|y-\tilde{x}\|)+2\|y-\tilde{x}\|.
 $$
\end{proposition}
\begin{proof}
Let $y\in B(\tilde{x},R)$. First note that assumption \eqref{KH}  implies that the result holds for $y=\tilde{x}$. Hence assume that  $y\neq\tilde{x}$.  From  definition of the  linearization error of $F$ in \eqref{eq:def.er} we have
$$
  -F(y)=-E(\tilde{x},y)-F(\tilde{x})-F'(\tilde{x})(y-\tilde{x}).
$$
Above equality together definition of sublinear mapping in \eqref{eq:dsblm}  give us
$$
 T_{\tilde{x}}^{-1}(-F(y))\supseteq  T_{\tilde{x}}^{-1}(-E(\tilde{x},y)) + T_{\tilde{x}}^{-1}(-F(\tilde{x})) + T_{\tilde{x}}^{-1}\left(-F'(\tilde{x})(y-\tilde{x})\right).
$$
Taking in account properties of the norm,  above inclusion implies
\begin{equation} \label{eq:bdfrt}
\left\|T_{\tilde{x}}^{-1}(-F(y))\right\| \leq   \left\|T_{\tilde{x}}^{-1}(-E(\tilde{x},y))\right\| + \left\|T_{\tilde{x}}^{-1}(-F(\tilde{x}))\right\| + \left\|T_{\tilde{x}}^{-1}\left(-F'(\tilde{x})(y-\tilde{x})\right)\right\|.
\end{equation}
Now we are going to bound the  three  terms in the left hand side of last inequality. Applying Lemma~\ref{pr:taylor} with $x=\tilde{x}$, $t=0$ and $s=\|y-\tilde{x}\|$ and taking in account that $f'(0)=-1$ we have
$$
 \left\|T_{\tilde{x}}^{-1}(-E(\tilde{x},y))\right\|\leq f(\|y-\tilde{x}\|)-f(0)+\|y-\tilde{x}\|.
$$
Definition of $T_{\tilde{x}}^{-1}$  implies that $-(y-\tilde{x})\in T_{\tilde{x}}^{-1}\left(-F'(\tilde{x})(y-\tilde{x})\right) $, hence we conclude that
$$
 \left\|T_{\tilde{x}}^{-1}\left(-F'(\tilde{x})(y-\tilde{x})\right)\right\| \leq \|y-\tilde{x}\|.
$$
Since assumption \eqref{KH} implies that second term in \eqref{eq:bdfrt} is bounded by  $f(0)$, thus substituting the two later inequalities into \eqref{eq:bdfrt} the desired inequality follows.
\end{proof}
\begin{proposition}  \label{pr:ar2}
Let $R>0$ and $ f:[0,R)\to \mathbb{R}$ a continuously differentiable function.   Suppose that $\tilde{x}\in \Omega$,   $f$ is a majorant function for $F$ at  $\tilde{x}$ and satisfies {\bf h4}. If  $0\leq \rho<  \beta/2$, where $\beta:=\sup\{ -f(t) ~:~ t\in[0,R)  \}$ then for any $\hat{x}\in  B(\tilde{x},  \rho)$ the scalar  function    $g:[0,R-\rho)\to \mathbb{R}$, 
\begin{equation*}
  g(t)=\frac{-1}{f'(\rho)}[f(t+\rho)+2\rho], 
\end{equation*}
is a majorant function for $F$ at $\hat{x}$ and also satisfies condition {\bf h4}.
\end{proposition}
\begin{proof}
Since the domain of $f$ is  $[0, R)$  and $f'(\rho)\neq 0$ (see Proposition~\ref{pr:maj.f} item iv ) we conclude that $g$ is well defined.   First we will prove that function $g$ satisfies condition {\bf h1, h2, h3} and {\bf h4}.  We trivially have that  $g'(0)=-1$.  Since $f$ is convex, combining this with {\bf h1} we have $f(t)+t\geq f(0)>0$,  for all  $0\leq t<R$, which   by using    Proposition~\ref{pr:maj.f} item {\it iv}   implies  $g(0)=  -[f(\rho)+2\rho]/f'(\rho)>0$, hence $g$ satisfies {\bf h1}. Using {\bf h2} and   Proposition~\ref{pr:maj.f} item {\it iv}   we easily conclude that $g$  also satisfies {\bf h2}.  Now,  as $\rho<  \beta/2$,  using Proposition~\ref{pr:maj.f}  items  {\it iii} and {\it iv}  we have
\[ 
\lim_{t\to \bar t-\rho}g(t)=\frac{-1}{f'(\rho)}(2\rho-\beta)<0\;, 
\]
which implies that  $g$ satisfies {\bf h4} and, as $g$ is continuous and $g(0)>0$,  it also satisfies  {\bf h3}.

To complete the proof, it remains to prove that $g$ satisfies \eqref{eq:MCAI}. First of all,  note that for any   $\hat{x}\in  B(\tilde{x},  \rho)$, from   Proposition~\ref{pr:maj.f}   item {\it iv}   we have  $\|\hat x - \tilde x\|<\rho<\bar t$ and by   using  Proposition~\ref{wdns} we obtain  that
\begin{equation} \label{eq:maj.02r1}
 \left\|T_{\hat{x}}^{-1}F'(\tilde{x})\right\|\leq \frac{-1}{f'(\rho)}.
\end{equation}
Because $ B(\tilde{x},  R) \subseteq \Omega$,  for any  $\hat{x}\in  B(\tilde{x},  \rho)$ we trivially  have   $B(\hat{x},R-\rho)\subset \Omega$.   Now, take $x,y\in \banacha$ such that
\[
x,y\in B(\hat{x},R-\rho),\qquad \|x-\hat{x}\|+\|y-x\|<R-\rho\;. 
\]
Hence $x,y\in B(\tilde{x},R)$ and $\|x-\tilde{x}\|+\|y-x\|<R$. Thus,  applying the inequality  of the Lemma~\ref{l:incltr} with $z=\hat{x}$ and $v=\tilde{x}$,  using  the property of the norm in \eqref{eq;pnor}, \eqref{eq:maj.02r1} and \eqref{eq:MCAI} we have
\begin{align*}
  \left\|T_{\hat{x}}^{-1}\left[F'(y)-F'(x)\right]\right\| \leq& \left\|T_{\hat{x}}^{-1}F'(\tilde{x})T_{\tilde{x}}^{-1}\left[F'(y)-F'(x)\right]\right\|\\
                                                                    \leq&  \left\|T_{\hat{x}}^{-1}F'(\tilde{x})\right\| \left\|T_{\tilde{x}}^{-1}\left[F'(y)-F'(x)\right]\right\|\\
                                                                   \leq &\frac{-1}{f'(\rho)}\left[f'\left(\|y-x\|+\|x-\tilde{x}\|)-f'(\|x-\tilde{x}\|\right)\right].
\end{align*}
On the other hand, since $f'$ is convex, the function $s\mapsto f'(t+s)-f'(s)$ is increasing for $t\geq 0$ and as $\|x-\tilde{x}\|\leq \|x-\hat{x}\|+\rho$ we conclude that
$$
 f'\left(\|y-x\|+\|x-\tilde{x}\|\right)-f'\left(\|x-\tilde{x}\|\right)  \leq  f'\left(\|y-x\|+\|x-\hat{x}\|+\rho\right) -f'\left(\|x-\hat{x}\|+\rho\right).
$$
Hence, combining the two above inequalities with the definition of the function  $g$ we obtain
\[
\|T_{\hat{x}}^{-1}\left[F'(y)-F'(x)\right]\| \leq  g'(\|y-x\|+\|x-\hat{x}\|)-g'(\|x-\hat{x}\|), 
\]
implying that the function  $g$ satisfies \eqref{eq:MCAI}, which complete the proof of the proposition.
\end{proof}
We are now ready to prove Theorem~\ref{th:knt2},   its proof is obtained by combining   Theorem~\ref{th:knt1} with  Propositions~ \ref{pr:ar1} and \ref{pr:ar2} .
\begin{proof}[Proof of Theorem~\ref{th:knt2}]
Since    $\hat{x}\in  B(\tilde{x},  \rho)$, from   Proposition~\ref{pr:maj.f}   item~{\it iv}   we have  $\|\hat x - \tilde x\|<\rho<\bar t$ and by   using  Proposition~\ref{wdns} we obtain  that
\begin{equation} \label{eq:maj.02r}
 \left\|T_{\hat{x}}^{-1}F'(\tilde{x})\right\|\leq \frac{-1}{f'(\rho)}\;,
\end{equation}
moreover,  the point  $\hat{x}$ satisfies the Robinson's Condition, namely, 
\begin{equation} \label{eq:rcr}
 \mbox{rge\,}T_{\hat{x}}=\banachb.
\end{equation}
Hence, applying  inclusion in  Lemma~\ref{l:incltr} with $z=\hat{x}$, $v=\tilde{x}$ and $w=-F(\hat{x})$, using property of the norm in \eqref{eq;pnor},  inequality \eqref{eq:maj.02r} and Proposition~\ref{pr:ar1} with $y=\hat{x}$ we obtain
\begin{align*}
 \left\|T_{\hat{x}}^{-1}(-F(\hat{x}))\right\| \leq &  \left\|T_{\hat{x}}^{-1}F'(\tilde{x})T_{\tilde{x}}^{-1}(-F(\hat{x}))\right\|\\
                                                                     \leq &  \left\|T_{\hat{x}}^{-1}F'(\tilde{x})\right\|  \left\|T_{\tilde{x}}^{-1}(-F(\hat{x}))\right\|\\
  \leq & \frac{-1}{f'(\rho)} \left[f\left(\|\hat x - \tilde x\|\right) +2\|\hat x - \tilde x\|\right].
\end{align*}
As $f'\geq -1$, the function $t\mapsto f(t)+2t$ is (strictly)
increasing.  Combining this fact with  $\|\hat x - \tilde x\|<\rho$,   the above inequality and definition of the function $g$ we conclude that
\[
\left\|T_{\hat{x}}^{-1}(-F(\hat{x}))\right\| \leq g(0), 
\]
Therefore, since   \eqref{eq:rcr} implies that $\hat{x}$ satisfies the Robinson's Condition and  Proposition~\ref{pr:ar2} implies that $g$ is a majorant function for $F$ at $\hat{x}$ satisfying condition {\bf h4},  the last inequality allow us  to  apply Theorem~\ref{th:knt1} for $F$ and the majorant function $g$ at point $\hat{x}$   for obtaining the desired result.
\end{proof}
\section{Special Case} \label{apl}
The affine majorant condition is crucial for our analysis.  It is worth pointing out that to construct a majorizing function for a given nonlinear function  is a very difficult problem and this is not our aim in this moment. On the other hand, there exist some classes of well known functions which  a majorant function is available, below we will present two examples, namely,   the classes of  functions  satisfying the a affine invariant Lipschitz-like and  Smale's  conditions, respectively.    In this sense,  the results obtained in Theorems~ \ref{th:knt1} and \ref{th:knt2}  unify the convergence  analysis  for the classes of inclusion problems involving these functions.
\subsection{Convergence result  under affine invariant  Lipschitz-like condition}
In this section, we will present a robust convergence theorem on Newton's method for solving nonlinear inclusion problem under affine invariant  Lipschitz-like condition, in particular, the result include as special  case the Theorem~$2$ of  \cite{Rob1972-2}.  Under  the Lipschitz-Like  condition, Theorem~\ref{th:knt1} becomes:
\begin{theorem}\label{th:kntlscl1}
Let $\banacha$, $\banachb$ be Banach spaces  and $\banacha$ reflexive, $\Omega\subseteq \banacha$ an open set and
  $F:{\Omega}\to \banachb$  a continuously
  differentiable function and  $C\subset \banachb $ a nonempty closed convex cone.   Take  $\tilde{x}\in \Omega$,
  $L>0$ and $b>0$.  Suppose  that $\tilde{x}\in \Omega$, $F$ satisfies  the Robson's Condition at  $\tilde{x}$,    
  $$
   B(\tilde{x},R)\subseteq \Omega, \qquad  \left\|T_{\tilde{x}}^{-1}\left[F'(y)-F'(x)\right]\right\| \leq  L\|y-x\|, \qquad \forall~ x,y\in B(\tilde{x},R),
  $$
  and $2bL\leq 1$. Moreover, assume that
  $$
  \left \|T_{\tilde{x}}^{-1}(-F(\tilde{x}))\right\|\leqslant b.
  $$
  Define,  the scalar  function $f:[0,+\infty)\to \mathbb{R}$ as 
$
  f(t):= Lt^2/2 - t + b.
$
Then  $f$ has   $ t_*:=(1-\sqrt{1-2bL})/L$ as  a smallest zero,  the sequences  generated by Newton's Method for solving   the inclusion 
  $
  F(x)\in C
  $ and  the equation  $f(t)=0$,  with starting
  point  $x_0=\tilde{x}$ and $t_0=0$, respectively, 
  \begin{align*}   \label{ns.KT}
    x_{k+1}\in x_k  +  \mbox{argmin} \left\{ \|d\|: \,  F(x_k)+F'(x_k)d \in C \right\},  \quad    t_{k+1} ={t_k}- \frac{f(t_k)}{f'(t_k)}, \qquad  k=0,1,\ldots\,, 
  \end{align*}
  are well defined, $\{x_k\}$ is contained in $B(\tilde{x},  t_*)$, $\{t_k\}$ is strictly increasing,   is contained in   $[0,t_*)$  and  converge  to  $t_*$ and    satisfy  the inequalities
\begin{equation*}\label{eq:bd}
  \|x_{k+1}-x_{k}\|   \leq  t_{k+1}-t_{k} , \qquad \qquad \|x_{k+1}-x_{k}\|\leq   \frac{t_{k+1}-t_{k}}{(t_{k}-t_{k-1})^2} \|x_k-x_{k-1}\|^2,
  \end{equation*}
 for all \( k=0, 1, \ldots\, , \) and \( k=1,2, \ldots\, \), respectively. Moreover,         $\{x_k\}$ converge  to  $x_*\in B[\tilde{x}, t_*]$ such that  $F(x_*)\in C,$  
\begin{equation*}\label{eq:002}
 \|x_*-x_{k}\|   \leq  t_*-t_{k}, \qquad \qquad t_*-t_{k+1}\leq \frac{1}{2}( t_*-t_{k}), \qquad k=0,1, \ldots\,
\end{equation*}
and, therefore,    $\{t_{k}\}$  converges $Q$-linearly to $t_*$ and   $\{x_k\}$   converge $R$-linearly to $x_*$.  Additionally, if  $2bL< 1$ then the following inequalities hold
\begin{equation*}\label{eq:bd2}
\|x_{k+1}-x_{k}\|\leq    \frac{ L}{2\sqrt{1-2bL}}   \|x_k-x_{k-1}\|^2,  \qquad t_{k+1}-t_{k} \leq \frac{ L}{2\sqrt{1-2bL}}  ({t_k}-t_{k-1})^2, \quad k=0,1, \ldots\,, 
  \end{equation*}
and, as  a consequence,    $\{x_k\}$ and  $\{t_k\}$  converge $Q$-quadratically
 to $x_*$ and $t_*$, respectively, as follow
 \begin{equation*}\label{eq:bd3}
\limsup_{k\to \infty}\frac{\|x_{*}-x_{k+1}\|} {\|x_*-x_{k}\|^2}\leq  \frac{ L}{2\sqrt{1-2bL}} , \qquad  \qquad t_{*}-t_{k+1} \leq \frac{ L}{2\sqrt{1-2bL}} ,  \quad k=0,1, \ldots\,.
  \end{equation*}
\end{theorem}
\begin{proof} 
 It easy  to see that $f:[0,+\infty)\to \mathbb{R}$
  defined by $ f(t):= Lt^2/2 - t + b$ is a majorant function
  to $F$ in $\tilde{x}$, with  the smallest zero equal to $t_*$. Therefore, the result follows from Theorem~\ref{th:knt1}.
\end{proof}
Under  the   affine invariant  Lipschitz-Like  condition,  the robust theorem on Newton's method for solving  nonlinear inclusion problems, namely, Theorem~\ref{th:knt2} becomes:
\begin{theorem}\label{th:kntlscl2}
Let $\banacha$, $\banachb$ be Banach spaces  and $X$ reflexive, $\Omega\subseteq \banacha$ an open set and
  $F:{\Omega}\to \banachb$  a continuously
  differentiable function and  $C\subset \banachb $ a nonempty closed convex cone.   Take  $\tilde{x}\in \Omega$,
  $L>0$ and $b>0$.  Suppose  that $\tilde{x}\in \Omega$, $F$ satisfies  the Robson's Condition at  $\tilde{x}$,    
  $$
   B(\tilde{x},R)\subseteq \Omega, \qquad \left\|T_{\tilde{x}}^{-1}\left[F'(y)-F'(x)\right]\right\| \leq  L\|y-x\|, \qquad \forall~ x,y\in B(\tilde{x},R),
  $$
  and $2bL< 1$. Moreover, assume that
  $$
  \left \|T_{\tilde{x}}^{-1}(-F(\tilde{x}))\right\|\leqslant b.
  $$
 Let $0\leq \rho<  (1-2Lb)/(4L)$  and  the quadratic   polynomial   $g:[0, + \infty)\to \mathbb{R}$ defined by
\begin{equation*} 
  g(t):=\frac{-1}{L\rho-1}\left[   \frac{L}{2}(t+\rho)^2 - (t+\rho) + b +2\rho\right].
\end{equation*}
Then the quadratic   polynomial   $g$ has the smallest zero given by
$$
 t_{*,\rho}:=\left(1-\rho L-\sqrt{1-2L(b-2\rho}\right)/L, 
 $$
 the sequences  generated by Newton's Method for solving   the inclusion 
  $
  F(x)\in C
  $ and  the equation  $g(t)=0$,  with starting  point $x_0=\hat{x}$ for any $\hat{x}\in  B(\tilde{x},  \rho)$ and $t_0=0$, respectively,
  \begin{align*}   
    x_{k+1}\in x_k  +  \mbox{argmin} \left\{ \|d\|: \,  F(x_k)+F'(x_k)d \in C \right\},  \quad    t_{k+1} ={t_k}- \frac{g(t_k)}{g'(t_k)}, \qquad  k=0,1,\ldots\,.
  \end{align*}  
  are well defined, $\{x_k\}$ is contained in $B(\tilde{x},  t_{*,\rho})$, $\{t_k\}$ is strictly increasing, is contained in   $[0,t_{*,\rho})$ and  converge  to  $t_{*,\rho}$ and   satisfy  the inequalities
  \begin{align*}  
 \|x_{k+1}-x_{k}\|   &\leq  t_{k+1}-t_{k}, \qquad &k=0, 1, \ldots\, , \\ \\
  \|x_{k+1}-x_{k}\|  &\leq   \frac{t_{k+1}-t_{k}}{(t_{k}-t_{k-1})^2} \|x_k-x_{k-1}\|^2\leq  \frac{L}{2\sqrt{1-2L(b-2\rho)}}  \|x_k-x_{k-1}\|^2, \qquad &k=1, 2, \ldots .
  \end{align*}  
 Moreover,    $\{x_k\}$ converge  to  $x_*\in B[\tilde{x}, t_{*,\rho}]$ such that  $F(x_*)\in C$,  satisfies the inequalities
\begin{equation*}
 \|x_*-x_{k}\|   \leq  t_{*,\rho}-t_{k}, \qquad \qquad t_{*,\rho}-t_{k+1}\leq \frac{1}{2}( t_{*,\rho}-t_{k}), \qquad k=0,1, \ldots\,
\end{equation*}
and the convergence of $\{x_k\}$ and  $\{t_k\}$ to  $x_*$ and $t_{*,\rho}$, respectively,  are  $Q$-quadratic as follows
 \begin{equation*}
\limsup_{k\to \infty}\frac{\|x_{*}-x_{k+1}\|} {\|x_*-x_{k}\|^2}\leq     \frac{L}{2\sqrt{1-2L(b-2\rho)}} ,  \quad \qquad t_{*,\rho}-t_{k+1} \leq   \frac{L}{2\sqrt{1-2L(b-2\rho)}}  ({t_{*,\rho}}-t_{k})^2,
  \end{equation*}
   for $k=0, 1, \ldots$. 
\end{theorem}
\begin{proof} 
The proof it follows from   Theorem~\ref{th:knt2} by noting that $f:[0,+\infty)\to \mathbb{R}$  defined by $ f(t):= Lt^2/2 - t + b$ is a majorant function to $F$ in $\tilde{x}$ and $\beta=(1-2Lb)/(4L)$. 
\end{proof}
\subsection{Convergence result under affine invariant Smale's condition}
In this section, we will present a robust convergence theorem on Newton's method for solving nonlinear inclusion problem under affine invariant  Smale's condition. For the degenerated  cone, i.e., $C=\{0\}$, this
is  the Corollary of Proposition 3 pp.~195 of \cite{S86}, see also Proposition 1 pp.~157 and Remark 1 pp.~158 of \cite{BCSS97}. 

First of all  we give a condition more easy to check then condition
\eqref{eq:MCAI}, when de functions under consideration are two times
continuously differentiable. For state the condition,  we need some definitions. Let $\banacha$,  $\banachb$  be a Banach spaces and  $\Omega\subseteq \banacha$ a open set. The norm of a  
$n$-th multilinear map $B:\banacha\times \ldots \times \banacha \to \banacha$  is defined by
  \[
 \|B\|:=\sup \left\{\|B(v_1, \dots, v_n) \| \;:\;\; v_1, \dots, v_n\in \banacha, \,\|v_i\|=1, \, i=1, \ldots, n \right\}.
 \]
 In particular, the norm of the $n$-th derivative of   $F:{\Omega}\to \banachb$   at a point  $x\in \Omega$ is given  by
 \[
 \|F^{(n)}(x)\|=\sup\left\{\|F^{(n)}(x)(v_1, \dots, v_n) \| ~:~ v_1, \dots, v_n\in \banacha, \,\|v_i\|=1, \, i=1, \ldots, n \right\}.
 \]
 Let $ T: \banacha \rightrightarrows  \banachb $   be sublinear mapping.  The   composition $TF^{(n)}(x):\banacha\times \ldots \times \banacha \rightrightarrows \banachb$ is defined by 
$
  TF^{(n)}(x)(v_1, \dots, v_n):=T(F^{(n)}(x)(v_1, \dots, v_n)).
$
Then,    for $F^{(n)}(x)(v_1, \dots, v_n) \in  \mbox{dom\,}T$ there hold:
\begin{equation*} \label{eq;pnorop}
  \| TF^{(n)}(x)\|=\sup \; \left\{ \left\|T(F^{(n)}(x)(v_1, \dots, v_n))\right\|~: ~v_1, \dots, v_n\in \banacha, \,\|v_i\|=1, \, i=1, \ldots, n \right\}, 
\end{equation*}
where $\|T(F^{(n)}(x)(v_1, \dots, v_n))\|:=\inf \{\|u\|~: ~u\in T(F^{(n)}(x)(v_1, \dots, v_n)) \}$.  

\begin{lemma} \label{lc}
Let $\banacha$  and $\banachb$  be a Banach spaces such that $ \banacha$ is reflexive,  $\Omega\subseteq \banacha$  and $F:{\Omega}\to \banachb$ a continuous function, two times continuously differentiable on $int(\Omega)$.      Suppose that   $\tilde{x}\in \Omega$ and     $\mbox{rge\,}T_{\tilde{x}}=\banachb$. If there exists  \mbox{$ f:[0,R)\to \mathbb{R}$} twice continuously differentiable such that
\begin{equation} \label{eq:lcec}
\|T_{\tilde{x}}^{-1}F''(x)\|\leqslant f''(\|x-\tilde{x}\|),
\end{equation}
for all $x\in \Omega$ such that $\|x-\tilde{x}\|<R$. Then $F$ and $f$
satisfy \eqref{eq:MCAI}.
\end{lemma}
\begin{proof}
Take  $x, y\in B(\tilde{x},R)$   such that $\|x-\tilde{x}\|+\|y-x\|<R$.  As  the ball is convex
  $
  x+\tau(y-x)\in B(\tilde{x},R),
  $
  for all $\tau\in [0,1]$.   Since, by assumption,  $\mbox{rge\,}T_{\tilde{x}}=\banachb$ we obtain   that $\mbox{dom\,}T_{\tilde{x}}^{-1}=\banachb$, hence properties of the norm and   assumption \eqref{eq:lcec} implies that
$$
  \left\|T_{\tilde{x}}^{-1}\left(F''(x+\tau(y-x))(y-x))\right)\right\|\leq   f''(\|(x-\tilde{x})+\tau(y-x)\|)\|y-x\|, \qquad \forall~ \tau\in [0, 1].
$$
Thus, as $\mbox{dom\,}T_{\tilde{x}}^{-1}=\banachb$   and  $\banacha$  is is reflexive,  we apply   Lemma~\ref{LiNg} with  $U=T_{\tilde{x}}^{-1}$,   $G(\tau)$ equal to the expression  in  parenthises on  the left had side of last inequality and $g(\tau)$ equal to the expression on the right hand site of that inequality,  obtaining   
$$
 \left\|T_{\tilde{x}}^{-1}\int_{0}^{1}F''(x+\tau(y-x))(y-x)d\tau\right\| \leqslant \int_{0}^{1}f''(\|(x-\tilde{x})+\tau(y-x)\|)\|y-x\|d\tau, 
$$
which, after performing the above integrations   yields the desired result.
 \end{proof}
\begin{theorem}\label{theo:Smale1}
Let $\banacha$, $\banachb$ be Banach spaces  and $\banacha$ reflexive, $\Omega\subseteq \banacha$ an open set and
  $F:{\Omega}\to \banachb$  a analytic  function and  $C\subset \banachb $ a nonempty closed convex cone.  Suppose  that $\tilde{x}\in \Omega$, $F$ satisfies  the Robson's Condition at  $\tilde{x}$,    
  \begin{equation*}\label{eq:SC}
    \gamma := \sup _{ k > 1 }\left\| \frac
    {T_{\tilde{x}}^{-1}F^{(k)}(\tilde{x})}{k !}\right\|^{\frac {1}{k-1}}< +\infty, 
  \end{equation*}
  $B(\tilde{x}, 1/\gamma)\subseteq \Omega$,  there exists  $b >0$ such that 
  \begin{equation*}
    \label{KH.2s}
 \left \|T_{\tilde{x}}^{-1}(-F(\tilde{x}))\right\|\leqslant b\,, 
  \end{equation*}
  and $\alpha:=b\gamma \leqslant 3-2\sqrt{2}.$   Define,  the analytic function $f:[0, 1/\gamma)\to \mathbb{R}$ as 
$
  f(t):= t/(1-\gamma t)-2t+b.
$
Then  $f$ has   
$$
 t_*:=\left(\alpha +1-\sqrt{(\alpha +1)^2-8\alpha}\right)/(4\gamma),
  $$ as  a smallest zero,  the sequences  generated by Newton's Method for solving   the analytics  inclusion $F(x)\in C $ and  the equation  $f(t)=0$,  with starting point  $x_0=\tilde{x}$ and $t_0=0$, respectively, 
  \begin{align*}   \label{ns.KT}
    x_{k+1}\in x_k  +  \mbox{argmin} \left\{ \|d\|: \,  F(x_k)+F'(x_k)d \in C \right\},  \quad    t_{k+1} ={t_k}- \frac{f(t_k)}{f'(t_k)}, \qquad  k=0,1,\ldots\,, 
  \end{align*}
  are well defined, $\{x_k\}$ is contained in $B(\tilde{x},  t_*)$, $\{t_k\}$ is strictly increasing, is contained in   $[0,t_*)$  and  converge  to  $t_*$ and    satisfy  the inequalities
\begin{equation*}\label{eq:bd}
  \|x_{k+1}-x_{k}\|   \leq  t_{k+1}-t_{k} , \qquad \qquad \|x_{k+1}-x_{k}\|\leq   \frac{t_{k+1}-t_{k}}{(t_{k}-t_{k-1})^2} \|x_k-x_{k-1}\|^2,
  \end{equation*}
 for all \( k=0, 1, \ldots\, , \) and \( k=1,2, \ldots\, \), respectively. Moreover,   $\{x_k\}$ converge  to  $x_*\in B[\tilde{x}, t_*]$ such that  $F(x_*)\in C,$  
\begin{equation*}\label{eq:002}
 \|x_*-x_{k}\|   \leq  t_*-t_{k}, \qquad \qquad t_*-t_{k+1}\leq \frac{1}{2}( t_*-t_{k}), \qquad k=0,1, \ldots\,
\end{equation*}
and, therefore,    $\{t_{k}\}$  converges $Q$-linearly to $t_*$ and   $\{x_k\}$   converge $R$-linearly to $x_*$.  Additionally, if  $\alpha < 3-2\sqrt{2}$ then the following inequalities hold
  \begin{align*}
  \|x_{k+1}-x_{k}\| &\leq     \frac{\gamma}{(1-\gamma t_*)[2(1-\gamma t_*)^2-1]}   \|x_k-x_{k-1}\|^2,  \qquad  &k= 0,1, \ldots\,,\\ \\
  t_{k+1}-t_{k}      & \leq   \frac{\gamma}{(1-\gamma t_*)[2(1-\gamma t_*)^2-1]}  ({t_k}-t_{k-1})^2, \quad   &k= 0,1, \ldots\,, 
\end{align*}
and, as  a consequence,    $\{x_k\}$ and  $\{t_k\}$  converge $Q$-quadratically
 to $x_*$ and $t_*$, respectively, as follow
  \begin{align*}
\limsup_{k\to \infty}\frac{\|x_{*}-x_{k+1}\|} {\|x_*-x_{k}\|^2} &\leq  \frac{\gamma}{(1-\gamma t_*)[2(1-\gamma t_*)^2-1]} ,\\ \\
t_{*}-t_{k+1}                                                                        &\leq  \frac{\gamma}{(1-\gamma t_*)[2(1-\gamma t_*)^2-1]} (t_*-t_{k}),   \qquad  &k= 0,1, \ldots\,,\\
\end{align*}
\end{theorem}
\begin{proof} 
  Use Lemma \ref{lc} to prove that $f:[0,1/\gamma) \to \mathbb{R}$
  defined by $ f(t)=t/(1-\gamma t)-2t+b,$ is a majorant function
  to $F$ in $\tilde{x}$, with root equal to $t_*$, see
  \cite{W99}.  Therefore, the result follows from Theorem~\ref{th:knt1}.
\end{proof}
Under  the affine invariant  Smale's   Condition,  the robust theorem on Newton's method for solving  nonlinear inclusion, namely, Theorem~\ref{th:knt2} becomes:
\begin{theorem}\label{theo:Smale2}
Let $\banacha$, $\banachb$ be Banach spaces  and $\banacha$ reflexive, $\Omega\subseteq \banacha$ an open set and
  $F:{\Omega}\to \banachb$  a analytic  function and  $C\subset \banachb $ a nonempty closed convex cone.  Suppose  that $\tilde{x}\in \Omega$, $F$ satisfies  the Robson's Condition at  $\tilde{x}$,    
  \begin{equation*}\label{eq:SC}
    \gamma := \sup _{ k > 1 }\left\| \frac
    {T_{\tilde{x}}^{-1}F^{(k)}(\tilde{x})}{k !}\right\|^{\frac {1}{k-1}}< +\infty, 
  \end{equation*}
  $B(\tilde{x}, 1/\gamma)\subseteq \Omega$,  there exists  $b >0$ such that 
  \begin{equation*} \label{KH.2s}
 \left \|T_{\tilde{x}}^{-1}(-F(\tilde{x}))\right\|\leqslant b\,, 
  \end{equation*}
  and $\alpha:=b \gamma < 3-2\sqrt{2}.$   Let $0\leq \rho< [\sqrt{2}(3-\alpha)-3]/(2\gamma\sqrt{2})$ and define   $g:[0, 1/\gamma-\rho)\to \mathbb{R}$ as 
$$
  g(t):= (t+\rho)/(1-\gamma (t+\rho))-2(t+\rho)+b+2\rho.
$$
Then  the scalar analytic   function   $g$ has a smallest zero given by
 $$
  t_{*,\rho}:=\left(\alpha +1-2\rho\gamma-\sqrt{(\alpha +1-2\rho\gamma)^2-8\alpha-8\rho\gamma(1-\alpha)}\right)/(4\gamma), 
 $$
 the sequences  generated by Newton's Method for solving   the analytics inclusion 
  $
  F(x)\in C
  $ and  equation  $g(t)=0$,  with starting  point $x_0=\hat{x}$ for any $\hat{x}\in  B(\tilde{x},  \rho)$ and $t_0=0$, respectively,
  \begin{align}   \label{ns.KT2}
    x_{k+1}\in x_k  +  \mbox{argmin} \left\{ \|d\|: \,  F(x_k)+F'(x_k)d \in C \right\},  \quad    t_{k+1} ={t_k}- \frac{g(t_k)}{g'(t_k)}, \qquad  k=0,1,\ldots\,.
  \end{align}  are well defined, $\{x_k\}$ is contained in $B(\tilde{x},  t_{*,\rho})$, $\{t_k\}$ is strictly increasing, is contained in   $[0,t_{*,\rho})$ and  converge  to  $t_{*,\rho}$ and   satisfy  the inequalities
\begin{align*}
 \|x_{k+1}-x_{k}\|   & \leq  t_{k+1}-t_{k}, \qquad \qquad \qquad    k=0, 1, \ldots\, ,  \\ \\
 \|x_{k+1}-x_{k}\|   &\leq   \frac{t_{k+1}-t_{k}}{(t_{k}-t_{k-1})^2} \|x_k-x_{k-1}\|^2\leq   \frac{\gamma}{(1-\gamma (t_{*,\rho}+\rho))[2(1-\gamma (t_{*,\rho}+\rho))^2-1]}  \|x_k-x_{k-1}\|^2, 
 \end{align*}
 for $k=1, 2, \ldots$.  Moreover,    $\{x_k\}$ converge  to  $x_*\in B[\tilde{x}, t_{*,\rho}]$ such that  $F(x_*)\in C$,  satisfies the inequalities
\begin{equation*}
 \|x_*-x_{k}\|   \leq  t_{*,\rho}-t_{k}, \qquad \qquad t_{*,\rho}-t_{k+1}\leq \frac{1}{2}( t_{*,\rho}-t_{k}), \qquad k=0,1, \ldots\,
\end{equation*}
and the convergence of $\{x_k\}$ and  $\{t_k\}$ to  $x_*$ and $t_{*,\rho}$, respectively,  are  $Q$-quadratic as follows
\begin{align*}
\limsup_{k\to \infty}\frac{\|x_{*}-x_{k+1}\|} {\|x_*-x_{k}\|^2} &\leq      \frac{\gamma}{(1-\gamma (t_{*,\rho}+\rho))[2(1-\gamma (t_{*,\rho}+\rho))^2-1]}, \\ \\
 \qquad t_{*,\rho}-t_{k+1}                                                   &\leq   \frac{\gamma}{(1-\gamma (t_{*,\rho}+\rho))[2(1-\gamma (t_{*,\rho}+\rho))^2-1]}   ({t_{*,\rho}}-t_{k})^2, \qquad k=0, 1, \ldots.
\end{align*}
\end{theorem}
\begin{proof} 
Use Lemma \ref{lc} to prove that $f:[0,1/\gamma) \to \mathbb{R}$  defined by $ f(t)=t/(1-\gamma t)-2t+b,$ is a majorant function
  to $F$ in $\tilde{x}$ (see \cite{W99}).   The proof it follows from   Theorem~\ref{th:knt2} by noting that $f$ is a majorant function to $F$ in $\tilde{x}$ and $\beta=[\sqrt{2}(3-\alpha)-3]/(2\gamma\sqrt{2})$.
\end{proof}
\section{Final remarks } \label{fr}
Let us finally make a few brief comments on  computational aspects of Newton's  method  for solving the nonlinear inclusion \eqref{eq:Inc}. Note that the first equality in \eqref{ns.KT} implies
$$
x_{k+1}={x_k}+d_k,\qquad F(x_k)+F'(x_k)d_k\in C,  \qquad k=0,1,\ldots.
$$
and for the degenerated cone $C=\{0\}$ the above iteration formally   becomes the Newton's iteration for solving nonlinear equation   $F(x)=0$, that is  
\begin{equation} \label{eq:ca1}
x_{k+1}={x_k}+d_k,\qquad F(x_k)+F'(x_k)d_k=0,  \qquad k=0,1,\ldots.
\end{equation}
Since the solution of the linear  equation \eqref{eq:ca1}   is computationally expensive, namely, at each iteration the deriavtive  at $x_k$ must be computed and stored.  To circumvent  drawbacks like this, Dembo, Eisenstat and Steihaug  introduced in \cite{DE1} the inexact Newton's method  for solving nonlinear equation $F(x)=0$. The inexact Newton's methods  for solving nonlinear equation $F(x)=0$ is any method which, given an initial point $x_0$, generates a sequence $\{ x_k\}$  as follows:  
$$
x_{k+1}={x_k}+d_k,\qquad \|F(x_k)+ F'(x_k)d_k\|\leq \theta_{k}\|F(x_{k})\|,  \qquad k=0,1,\ldots,
$$
for a suitable forcing sequence $\{\theta_{k}\}$,   which is used to control the level of accuracy. Hence, solutions of practical problems are obtained by computational implementations  of the inexact Newton's methods. Therefore,  we  extend the inexact Newton's methods   for solving nonlinear inclusion as any method which, given an initial point $x_0$, generates a sequence $\{ x_k\}$  as follows:  
$$
x_{k+1}={x_k}+d_k,\qquad   d(0, F(x_k)+ F'(x_k)d_k-C)\leq \theta_{k}d(0, F(x_{k})-C), \qquad k=0,1,\ldots,
$$
 for a suitable forcing sequence $\{\theta_{k}\}$, where  $ d(x, D)$ denotes the distance from a point $x\in \banacha$ to the subset $D\subset  \banachb$; that is   $d(x, D)=\inf\{\|x-x'\| ~: ~x'\in D\}$.  The analysis of  these methods  under  Lipchitz-like and   majorant conditions will be done in the future.
\section*{Acknowledgment}

This work was done while the  author were  visiting  The Institute of Computational Mathematics and Scientific/Engineering Computing, Academy of Mathematics and Systems Science at Chinese Academy of Science. He is  grateful to host institution for the congenial
scientific atmosphere that it has provided during his visit.

\end{document}